\documentclass[a4paper]{article}

\usepackage{hyperref}
\usepackage{mynotation,a4,amsthm,amsfonts,amssymb,amsmath,amssymb,pb-diagram,amscd,enumerate,mathrsfs,eucal,comment}
\usepackage[alphabetic,nobysame]{amsrefs}
\usepackage{color}
\newtheorem{thm}{Theorem}[section]
\newtheorem{theorem}[thm]{Theorem}

\newtheorem{cor}[thm]{Corollary}
\newtheorem{lem}[thm]{Lemma}
\newtheorem{lemma}[thm]{Lemma}

\newtheorem{proposition}[thm]{Proposition}

\theoremstyle{definition}
\newtheorem{defi}[thm]{Definition}
\newtheorem{definition}[thm]{Definition}

\theoremstyle{remark}

\newtheorem{remark}[thm]{Remark}

\numberwithin{equation}{section}

\newcommand{\tensor}{\otimes}
\newcommand{\complexs}{\mathbb{C}}
\newcommand{\integers}{\mathbb{Z}}
\newcommand{\naturals}{\mathbb{N}}
\newcommand{\reals}{\mathbb{R}}
\newcommand{\innerprod}[1]{\langle#1\rangle}

\renewcommand{\epsilon}{\varepsilon}

\renewcommand{\mathcal}{\mathscr}

\begin{document}
\title{Codimension two index obstructions to\\ positive scalar curvature}

\author{
Bernhard Hanke
\phantom{X}Daniel Pape\thanks{the author was supported by the German Research Foundation (DFG) through
the Research Training Group 1493 \textquoteleft Mathematical structures in modern quantum physics\textquoteright. \texttt{www.uni-math.gwdg.de/pape}}\phantom{X}
Thomas Schick\thanks{partially funded by the Courant Research Center \textquoteleft Higher order structures in Mathematics\textquoteright\ within the German initiative of excellence. \texttt{www.uni-math.gwdg.de/schick}}\\
\small Universit\"{a}t Augsburg\\[-0.8ex]
\small Universit\"{a}tsstrasse 14, 86159 Augsburg, Germany\\
\small Georg-August-Universit\"{a}t G\"{o}ttingen\\[-0.8ex]
\small Bunsenstr. 3, 37073 G\"{o}ttingen, Germany\\
\small \texttt{hanke@math.uni-augsburg.de}\phantom{X}
\small \texttt{pape@uni-math.gwdg.de}\phantom{X}
\small \texttt{schick@uni-math.gwdg.de}
}

%\date{Feb 1, 2011}
\date{}
\maketitle

\begin{abstract}

We derive a general obstruction to the existence of 
 Riemannian metrics of positive scalar curvature on closed spin manifolds in
 terms of 
 hypersurfaces of codimension
 two. The proof is based on coarse index
 theory for Dirac operators that are twisted with  Hilbert $C^*$-module bundles.

 Along the way we give a complete and self-contained proof that the minimal 
 closure of a Dirac type operator twisted with a 
 Hilbert $C^*$-module bundle on a 
 complete Riemannian manifold is a regular and  self-adjoint operator 
 on the Hilbert $C^*$-module of $L^2$-sections of this bundle. 
 
Moreover, we give a new proof of Roe's vanishing theorem for the
 coarse index of the Dirac operator on a complete Riemannian manifold  whose 
 scalar curvature is uniformly positive outside of  a compact subset. This proof  immediately generalizes to Dirac operators 
  twisted with Hilbert $C^*$-module bundles. 
  
\end{abstract}

\section{Introduction}

A central theme of geometric topology in recent decades asks whether a
given smooth manifold admits a Riemannian metric with positive scalar
curvature. On spin manifolds  the most powerful obstructions to existence of
such  metrics are based on index theory for the Dirac operator. Indeed the
Schr\"odinger-Lichnerowicz formula \cite{Schroedinger} implies that on a  spin manifolds  with uniformly positive scalar curvature the Dirac operator is invertible and hence its index, suitably defined if the manifold is not compact, has to vanish. 

Rosenberg \cites{Rosenberg_PSC_NC_I, Rosenberg_PSC_NC_II, Rosenberg_PSC_NC_III}
used Dirac operators twisted with flat Hilbert $C^*$-module bundles whose
indices
lie in the K-theory of $C^*$-algebras in order to obtain
particularly strong obstructions to the existence of positive scalar curvature metrics. 
In particular, using the Mishchenko bundle, the canonical flat
$C^*\pi_1(M)$-bundle on $M$, one obtains the Rosenberg 
index obstruction $\alpha(M)\in K_*(C^*\pi_1(M)))$.

 Roe \cite{Roe_AMS} 
developed coarse index theory to define meaningful indices of 
Dirac operators on non-compact complete manifolds. This can also be used
to gain interesting information for compact manifolds by passing to 
non-compact covering spaces.

Gromov and Lawson \cite[Theorem 7.5]{MR720933} found  an intriguing
obstruction to positive scalar curvature based on submanifolds of codimension two: 
if $M$ is a closed aspherical spin 
manifold with a hypersurface $N$ of codimension two with trivial normal bundle
such that $N$ is
enlargeable and $\pi_1(N)$ injects into $\pi_1(M)$, then $M$ does not admit a
Riemannian metric of positive scalar curvature. 

The main purpose of this paper is to illuminate this result 
from an index theoretic perspective. Our proof is based on coarse index theory
for Dirac operators 
twisted with Hilbert $C^*$-module bundles. This allows us to prove the following 
statement, where in particular the asphericity of $M$  in
\cite{MR720933} is weakened to the vanishing of the second homotopy group.

\begin{theorem}\label{theo:codim2}
  Let $M$ be a closed connected spin manifold with $\pi_2(M)=0$. Assume that
  $N\subset M$ is a codimension two submanifold with trivial normal bundle and
  that the induced map $\pi_1(N)\to \pi_1(M)$ is injective. Assume that the
  Rosenberg index of $N$ does not vanish: $0\ne \alpha(N)\in
  K_*(C^*\pi_1(N))$. 

  Then $M$ does not admit a Riemannian metric of positive scalar curvature.
\end{theorem}

\begin{remark}
   Here and above one can use either the
reduced or the maximal group $C^*$-algebra. The first  one is closely connected to
the Baum-Connes and the strong Novikov conjecture, but a priori the latter one might 
lead to stronger obstructions. The material in the paper at hand 
is independent of which group $C^*$-algebra is used so that we 
will not distinguish them in our notation. 
\end{remark}

\begin{remark}
  In \cite[Theorem 3.4]{MR2929034} a result close to our Theorem
  \ref{theo:codim2} is stated without any assumption on
  $\pi_2(M)$. Unfortunately, the statement of \cite[Theorem 3.4]{MR2929034} is
  wrong, the manifolds  $N=T^n$ and $M=(T^n\times S^2) \# (T^n\times S^2)$, 
  $n>2$, providing counterexamples. Our correct formulation of Theorem
  \ref{theo:codim2} had been established long before \cite{MR2929034} appeared, 
  and the authors of the present paper reported on it at several occasions in
  seminars and at conferences.% at the University of M\"unster in 2005. 
\end{remark}

\begin{remark}
  The concept of enlargeability is not used in our paper; it
  is entirely based on properties of the
  Rosenberg index. Because  enlargeable spin manifolds 
  have non-vanishing
  Rosenberg index \cite[Theorem 1.2]{MR2353861}, \cite{MR2259056},
  \cite[Theorem 1.5]{HankeKotschickRoeSchick}, this is no loss of 
  generality. 
  \end{remark}

Coarse index theory as developed by Roe is based on 
functional calculus for  the (unbounded) Dirac operator. In our context 
we are dealing with Dirac operators twisted by Hilbert
$C^*$-module bundles so that, in order to apply  the functional calculus 
in \cite{MR1949157}, it is required to establish 
 regularity and self-adjointness of their closures. In our 
 opinion this fact  is not well-documented in the literature and 
 hence we decided to give a self-contained and
complete proof of the following theorem, which might be of independent interest. 
{Throughout our paper, $A$ denotes a complex unital
  $C^*$-algebra}.

\begin{theorem}\label{theo:reg_sa}
  Let $M$ be a complete Riemannian manifold, and let $E\to M$ be 
  a smooth Hilbert $A$-module bundle with finitely generated projective
  fibers, which is equipped  with a connection $\nabla$, compatible with the
inner product.

  Let $D$ be any Dirac type operator on a  {complex} Dirac bundle 
  $S\to M$ and $D_E$ its twist with
  $(E,\nabla)$. Then the closure of $D_E$ is a {densely defined, 
  regular and self-adjoint operator} on the
  Hilbert $A$-module $L^2(M,S \otimes E)$ of $L^2$-sections of $S\tensor E$.
\end{theorem}

We follow the program of Vassout \cite{Vassout} who proves a corresponding
statement for foliations, based on the existence of a suitable
pseudodifferential calculus.
\begin{remark}
  We could locate a couple of accounts of the result for compact manifolds,
  which however, for our taste, were quite sketchy and did not cover
  the case of non-compact manifolds. Zadeh \cite[Lemma 2.1]{MR2670972} offers
  an alternative proof that also includes the case of non-compact manifolds. However, it is based on
  properties of the wave operators $e^{itD_E}$,  whose existence is 
  assumed in \cite{MR2670972} without further reference. {We believe that  
  a construction of these operators is possible, independent of a general 
  functional calculus for $D_E$ (which would depend on normality and
  self-adjointness of this operator and therefore would render the argument
  circular). But we could not  find a detailed construction in the
  literature.} Therefore we decided to give a full and independent proof of
  Theorem \ref{theo:reg_sa} in Section \ref{regular} below. 
\end{remark} 

A final crucial ingredient of our proof of the codimension two obstruction
in Theorem~\ref{theo:codim2} is a
generalized vanishing theorem for the coarse index on non-compact manifolds.

\begin{theorem}[Partial vanishing theorem]
\label{thm:Vanishing_theorem_intro}
Let $(M,g)$ be a complete connected non-compact Riemannian spin manifold such
that, outside of a compact subset, the  scalar curvature  is uniformly
positive. Let $E\to M$ be a  Hilbert $A$-module bundle {as in  Theorem 
\ref{theo:reg_sa} above} and assume that this bundle is flat.  Then the coarse
index 
$\ind(D_{E})\in K_*(C^*(M;A))$ vanishes. 
\end{theorem}

\begin{remark}
  The special case of this result with $A=\complexs$ and trivial $E$ has
  been stated in \cite[Proposition 3.11 and following remark]{Roe_AMS} without
  proof. Only  recently, Roe \cite{Roe_partial_vanish} published a full
  proof of this special case, using the theory of Friedrichs extensions of
  unbounded operators. Zadeh \cite[Theorem 3.1]{MR2929034} offers a proof
  of Theorem \ref{thm:Vanishing_theorem_intro}, again based on Friedrichs 
  extensions. We feel that this is not completely satisfactory. Although
  the concept of Friedrichs extensions for unbounded operators on Hilbert
  $A$-modules should exist, it has not been developed yet, to the 
  best of our knowledge. In particular the regularity of the resulting operator 
  must be taken care of. 
\end{remark}

We present a proof in the spirit of Roe's coarse index
theory, based on functional calculus and unit propagation
of the wave operator. This proof first appeared  in the second author's doctoral
thesis \cite[Theorem 0.2.1]{thesis_pape}. We expect that it can be generalized
to other interesting 
situations, notably to perturbations of the signature operator, as they
show up in proofs of the homotopy invariance of higher signatures, compare \cite{MR1142484}.

\begin{remark}
  The codimension two obstruction in \ref{theo:codim2} has a slight
  strengthening: even \emph{stably} $M$ does not admit a metric of positive
  scalar curvature. Here, ``stably'' means that for every simply connected
  closed $8$-dimensional spin manifold $B$ with $\hat A(B)=1$, i.e.~for any
  so-called \emph{Bott manifold}, and for every $l\ge 0$ the manifold $M\times
  B^l$ does not admit a metric of positive scalar curvature. This simply
  follows by applying the codimension two obstruction theorem to  $N\times
  B^l \subset M\times B^l$. 

  The \emph{stable Gromov-Lawson-Rosenberg conjecture} \cite[Conjecture
  4.17]{RosenbergStolz} states that a closed spin manifold $M$ 
  stably admits a metric with positive scalar curvature if and only if its
  Rosenberg index $\alpha(M)\in KO_*(C_{\reals}^*\pi_1(M))$ vanishes. 

  Stolz \cites{MR1403963,MR1937026} proved that the stable Gromov-Lawson-Rosenberg conjecture holds for all manifolds whose fundamental groups satisfy the strong Novikov conjecture.
 Recall that the \emph{unstable} version of
  \cite[Conjecture 4.8]{RosenbergStolz} is not true
  \cite{MR1632971}. The construction of a corresponding example uses the
  \emph{codimension one} obstruction of Schoen and Yau \cite{MR535700}, which is 
  based on minimal hypersurfaces and independent from  index theory.  

 Arguing  in a rather indirect manner using Stolz's theorem 
  it follows that under the assumptions  of
 \ref{theo:codim2} {not only $\alpha(N)$}, but also the
 Rosenberg index $\alpha(M)\in K_*(C^*\pi_1(M))$
 is non-zero, 
if $\pi_1(M)$ satisfies the strong
   Novikov conjecture,
   
 However, we have not been able to prove non-vanishing 
 of the Rosenberg index $\alpha(M)$ in complete generality in the situation of
 Theorem  \ref{theo:codim2}.   
 We leave this as an open question. In view of the possibility that
 $\alpha(M)$ could be zero,  
 one might speculate whether Theorem \ref{theo:codim2} 
 can be used in the end to establish counterexamples to the strong Novikov 
 conjecture.
 \end{remark}

\begin{remark}
  We formulate and prove our theorem in the context of  complex $C^*$-algebras and the
  complex Dirac operator. Firstly,  this is most suited to the approach to
  coarse index theory as developed by Roe, and secondly the literature on self 
  adjoint regular
  operators and their functional calculus is much more complete in this
  case. Nonetheless, we expect that all of our results can be generalized to real
  $C^*$-algebras and the real Dirac operator, which indeed furnish the most efficient 
  context for geometric applications of the index theory of Dirac operators.
\end{remark}

\section{Regularity and self-adjointness of Dirac operators twisted with Hilbert-module bundles} 
\label{regular}

In this section we prove in detail that twisted Dirac type operators on
complete Riemannian manifolds have regular and self-adjoint closures. As a preparation,
we recall some basics about Hilbert $C^*$-modules and regular, self-adjoint
operators on them.

%\subsection{Criterion for regularity and self-adjointness}

Let $A$ be a (complex) $C^*$-algebra, which in our paper is assumed to be unital 
throughout. Recall
that a Hilbert $A$-module is a right $A$-module with an $A$-valued inner product
satisfying a number of axioms, see page 4 in \cite{MR1325694}, that serves as
our main reference for the theory of Hilbert $C^*$-modules and unbounded operators. 
We emphasize that, in contrast to usual Hilbert spaces, a closed submodule
$\mcH_0\subset\mcH$ and the orthogonal submodule $\mcH_0^{\perp}\subset\mcH$
do not complement each other in general, i.e.~usually
$\mcH_0\oplus\mcH_0^{\perp}\varsubsetneq\mcH$. If the opposite inclusion holds
one says that $\mcH_0$ \emph{has an orthogonal complement}. 

Let $\mcH_1$ as well as $\mcH_2$ be Hilbert
$A$-modules. An \emph{operator} from $\mcH_1$ to $\mcH_2$ is an $A$-linear
map $T\colon\dom{T}\pf\mcH_2$ on a submodule $\dom{T}$ of $\mcH_1$. The latter
is called the \emph{domain} of $T$. One calls $T$ \emph{densely defined} if
$\overline{\dom{T}}=\mcH_1$. An operator $S$
from $\mcH_1$ to $\mcH_2$ is called an \emph{extension} of $T$, written
$T\subset S$, if $\dom{T}\subset\dom{S}$ and $Tx=Sx$ holds for each
$x\in\dom{T}$. The \emph{graph} of $T$ is denoted by $G(T)$:
\[
G(T):=\{(x,y)\in\mcH_1\oplus\mcH_2\,;\,x\in\dom{T}\textnormal{ and } y=Tx\}\text{ .}
\]
One calls $T$ \emph{closed} if $G(T)$ is closed in $\mcH_1\oplus\mcH_2$.  The
operator $T$ is called {\emph{closable}} if it admits a closed extension. This
is equivalent to the existence of an operator $S$ with
$G(S)=\overline{G(T)}$.  In this case $S$ is the smallest closed extension of $T$, 
called the \emph{closure} and usually denoted $\overline{T}$ or
$T_{\min}$. Its domain is
\begin{equation*}
\label{eq:Describtion of the minimal domain}
\dom{T_{\min}}
=
\{x\in\mcH_1\,|\,\textnormal{$\exists\,(x_n)_{n\in\mbN}\subset\dom{T}$ with $x_n\pf x$ and $Tx_n\pf T_{\min}x$}\}\text{ .}
\end{equation*}

For a densely defined operator $T\pp\dom{T}\pf\mcH_2$ we set
\[
\dom{T^*}
:=
\{y\in\mcH_2\,|\,\exists  z\in\mcH_1\textnormal{ with } \langle Tx,y\rangle=\langle x,z\rangle\;\forall  x\in\textnormal{dom}(T)\}
\text{ .}
\]
The element $z$  appearing on the right is unique and we can define the \emph{adjoint
  operator} of $T$ as the operator $T^*\pp\dom{T^*}\pf\mcH_1$ given by
$T^*y=z$. Note that $T^*$ is a closed operator {and that for a closable 
operator $T$, we have $T^* =  \overline{T}^*$.} 

The operator $T$ is called {{\em adjointable}}, if $\dom{T}=\mcH_1$ and $\dom{T^*}=\mcH_2$. 
Adjointable operators are automatically bounded, but bounded operators are not 
necessarily adjointable, see \cite[p. 8]{MR1325694}. {Because every 
densely defined operator has an adjoint by definition, we will prefer 
 the term {\em bounded adjointable} instead of {\em adjointable} in order to 
avoid any confusion}. The space of 
bounded adjointable operators is denoted  $\mathcal{L}_A(\mcH_1,\mcH_2)$, or briefly
$\mathcal{L}_A(\mcH_1)$ if $\mcH_1=\mcH_2$. 

 The subspace of \emph{$A$-compact operators} is the closure of
  the $A$-linear span of operators of the form $x\mapsto \innerprod{x,a}\,b$ where 
  $a\in \mcH_1$, $b\in\mcH_2$.

\begin{defi}
\label{defi;Regular operator}
Let $T\pp\dom{T}\pf\mcH_2$  be an operator with $\dom{T}\subset\mcH_1$. One
calls $T$ \emph{regular} if
\begin{enumerate}[(i)]
\item $T$ is densely defined and closed,
\item $T^*$ is densely defined,
\item the graph of $T$ (a closed subset of $\mcH_1\oplus \mcH_2$) has an
  orthogonal complement.
\end{enumerate}
\end{defi}

We now come to a useful criterion for regularity and self-adjointness. Recall 
that a densely defined operator $T$ is called {\em symmetric}, if $T  \subset
T^*$ 
and {\em self-adjoint}, if $T = T^*$. Because $T^*$ is  a closed operator
by
\cite[p.~95]{MR1325694},
symmetric operators are closable and 
self-adjoint operators are closed. 
\begin{thm}[Characterization of self-adjoint, regular operators]
\label{thm:Characterization of self-adjoint, regular operators}
Let $T$ be a closed, densely defined and symmetric operator in the
Hilbert $A$-module $\mcH$. Then the  following are equivalent:
\begin{enumerate}[(i)]
\item $T$ is self-adjoint and regular,
\item $T+ i$, $T-i$ both have dense range.
\end{enumerate}
\end{thm}
\begin{proof}
  By \cite[Lemma 9.8]{MR1325694}, if $T$ is self-adjoint and regular then
  $T\pm i$ both have dense range.

  Conversely, assume $T\pm i$ both have dense range.
  By \cite[Lemma 9.7]{MR1325694}, the assumptions imply that $T+i$ and $T-i$
  are injective and have closed range. Therefore $T\pm i$ both are bijective (and 
  in particular  both operators have range $\mathcal{H}$). 

  As $T$ is symmetric, $T^*$ is an extension of $T$, and therefore $T^*\pm i$
  are extensions of $T\pm i$. As already $T\pm i$ is surjective, $T^*$ is a
  proper extension of $T$ if and only if both operators $T^*\pm i$ have non-trivial
  kernel. But for $x\in \ker{T^*+i}$ and $y\in \dom{T}$ we have
  \begin{equation*}
    0=\innerprod{(T^*+i)x,y} = \innerprod{x,(T-i)y},
  \end{equation*}
  and since $T-i$ is surjective, $x=0$. Therefore $T^*=T$, i.e.~the assumption
  implies that $T$ is self-adjoint. Finally \cite[Lemma 9.8]{MR1325694}
  implies that $T$ is also  regular.
\end{proof}

\subsection{Regularity and self-adjointness of twisted Dirac operators}
\label{sec:regularity}

Let $(M,g)$ be a complete Riemannian manifold, {let $S \to M$ be a complex 
Dirac bundle on $M$} and let 
$D\colon\Gamma^{\infty}(M,S)\pf\Gamma^{\infty}(M,S)$  be the corresponding
 Dirac type operator
acting on the sections of $S$, see
\cite[Definition 3.4]{MR1670907} or \cite[Definition II.5.2]{MR1031992}. The
main examples we have in mind are the Dirac operator of a Riemannian spin
manifold, the de Rham operator of a general Riemannian manifold, the signature
operator of an oriented Riemannian manifold, or the Dolbeault
operator of a K\"{a}hler manifold. 

In addition, let $A$ be a unital complex $C^*$-algebra and $E$ a smooth Hilbert $A$-module bundle whose 
fibers are  finitely generated projective Hilbert $A$-modules, equipped with
a metric connection $\nabla^E$. We obtain the twisted Dirac 
operator $D_E$ acting on smooth sections $\Gamma^{\infty}(M,S\otimes
  E)$. Note that the bundle $S\otimes 
E \to M$ inherits the structure of a Hilbert $A$-module bundle so that the  Riemannian
metric on $M$  allows us to define an $A$-valued inner product on the
space $\Gamma^{\infty}_{\text{cpt}}(M,S\otimes E)$ of compactly supported
smooth sections of 
this bundle. This inner product is given by the formula
\[
    \langle s_1,s_2\rangle :=\int_M \langle s_1(x),s_2(x)\rangle_{S_x\otimes
  E_x}\,d\lambda_g(x),
\]
where $\lambda_g$ is the measure associated with
$g$. The corresponding completion is the
Hilbert $A$-module $L^2(M, S\otimes E)$, by definition. 

\begin{thm}
\label{thm:Essential self-adjointness of twisted Dirac operator}
Let $(M,g)$ be a complete Riemannian manifold and let  $(E,\nabla^E)$ be  a  smooth
finitely generated 
projective Hilbert $A$-module bundle with metric connection, then
\begin{equation*}
D_E\colon\Gamma^{\infty}_{\textnormal{cpt}}(M,S\otimes E)\pf\Gamma^{\infty}_{\textnormal{cpt}}(M,S\otimes E)
\end{equation*}
is closable in $L^2(M, S\otimes E)$ and the minimal closure is regular and
self-adjoint as unbounded 
Hilbert $A$-module operator. It is the unique self-adjoint extension of $D_E$.
\end{thm}

\begin{proof}
Recall from the proof of \cite[Proposition 3.11]{MR1670907}
% or~\cite[Chapter 2, Proposition 5.3]{MR1031992}
that 
$D_E$ with domain equal to $\Gamma_{\mathrm{cpt}}^{\infty}(M,S\otimes E)$ is
symmetric and hence closable.

\smallskip
We first deal with the case when $M$ is compact. In this context,
Mish\-chen\-ko and 
Fomenko~\cite{MR548506} developed a pseudodifferential calculus for
operators on smooth sections of $S\otimes E$ which the
following properties, among others: 

\begin{enumerate}[(i)]
\item The identity is an operator of order $0$ in the pseudodifferential
  calculus. 
\item The operator $D_E$ is an operator of order $1$ in the pseudodifferential
  calculus. 
\item The operator $D_E$ is has a parametrix, i.e.~there are
  operators $Q$ of 
order $-1$ and $R,T$ of order $-\infty$ in the calculus such that $D_EQ=1-R$ and
$QD_E=1-T$.
\item Each operator $P$ of order $\le 0$ in the pseudodifferential calculus
  extends (uniquely) to a bounded adjointable operator $\overline{P}$ on
  $L^2(M, S\otimes E)$.
 \item If $P$ is an operator of order $<0$ in the calculus then its bounded
   adjointable extension is an $A$-compact operator.
\item Each operator in the calculus has a formal adjoint {of the same 
order}, for $D_E$ the 
  formal adjoint is $D_E$.
\item If $P$ is an operator of order $\le 0$ in the calculus, then its adjoint,
 {which is necessarily equal to the adjoint 
  of its closure}, is the closure of its formal adjoint.
\end{enumerate}

Now it turns out that these are exactly  the properties
needed for a proof of regularity and self-adjointness of the
closure of $D_E$, first given, to our knowledge, in~\cite[Proposition
3.4.9]{Vassout}. Alternatively, in~\cite[Section 1]{MR1235284} an
argument due to Skandalis for the statement is sketched which even works in
the case of Lipschitz manifolds. As the thesis \cite{Vassout} is only
available on a university homepage without permanent link, we repeat this
proof here for the reader's convenience.

It is based on the following three properties:
\begin{enumerate}[(i)]
\item\label{item:comp} $\overline{D_E}\circ\overline{Q}= \overline{D_E\circ
    Q}$ and 
  $\overline{D_E}\circ\overline{T} =\overline{D_E\circ T}$, and these operators are bounded
  (hence everywhere defined),
\item\label{item:dom} $\dom{\overline{D_E}}=\im{\overline{Q}}+\im{\overline{T}}$,
\item\label{item:selfad} {$\overline{D_E}=(D_E)^*$}, in particular $\overline{D_E}$ is self-adjoint.
\end{enumerate}

To establish  $\overline{D_EQ}\subset
\overline{D_E}\circ \overline{Q}$ let $x$ be in the domain of $\overline{D_EQ}$. By
definition, this means there are smooth sections $x_n$
converging to $x$ such that $D_EQ(x_n)$ converges to
$y:=\overline{D_EQ}(x)$. Now, $Q$ has negative order and therefore a bounded
closure $\overline{Q}$ and hence $\lim_{n\to\infty}Qx_n=
z=\overline{Q}(x)$. By definition of $\overline{D_E}$, $z$ is in the domain of
$D_E$ and $\overline{D_E}(z)=y$. So, indeed $\overline{D_EQ} \subset
\overline{D_E}\circ\overline{Q}$. Now $D_EQ$ has order zero and
therefore its closure is bounded, in particular everywhere defined, so that
$\overline{D_EQ}$ has no proper extension and therefore we have the required
equality. 

This argument also shows that
$\overline{D_ET}=\overline{D_E}\circ\overline{T}$. This finishes the proof of \eqref{item:comp}.

For \eqref{item:dom}, if $x$ is in the domain of $\overline{D_E}$ then by
definition there are smooth sections $x_n$ such that
$\lim_{n\to\infty}x_n= x$ and $\lim_{n\to\infty}D_Ex_n=
y= \overline{D_E}(x)$. Then $QD_E x_n = x_n - Tx_n$ where $Q$ and $T$ have
continuous closures, and hence, passing to the limits
\begin{equation*}
  \overline{Q}(\overline{D_E}(x)) = x -\overline{T}(x)
\end{equation*}
or in other words $x = \overline{Q}(y)+\overline{T}(x) \in
\im{\overline{Q}}+\im{\overline{T}}$. 

Conversely, \eqref{item:comp} implies that $\im{\overline{Q}}\subset
\dom{\overline{D_E}}$ and $\im{\overline{T}}\subset\dom{\overline{D_E}}$. 

To prove \eqref{item:selfad}, the self-adjointness of the operator
$\overline{D_E}$, recall that by symmetry $\overline{D_E}\subset (D_E)^*$. For
the converse inclusion, recall that $D_EQ = 1-R$ and therefore $(D_EQ)^* =
1-R^*$. For adjoints of compositions one always has $Q^*D_E^*\subset (D_E
Q)^*$. Because $Q^*$ is the closure of the formal adjoint of $Q$, which is
bounded as it is of negative order, $\dom{D_E^*}=\dom{Q^*D_E^*}$ and for $x\in
\dom{D_E^*}$ we have (using $Q^*D_E^*\subset (D_EQ)^*=1-R^*$)
\begin{equation*}
  Q^*(D_E^*(x)) +R^*x= x,
\end{equation*}
so $\dom{D_E^*}\subset \im{Q^*}+\im{R^*}$. 

Taking the formal adjoint of the parametrix equations $D_EQ=1-R$, $QD_E=1-T$,
we see that also the formal adjoint of $Q$ is a parametrix of $D_E$, with
error terms the formal adjoints of $R,T$, {but with the roles of $R$ and $T$ exchanged in the parametrix equations. Because the 
closures of the formal adjoints of $Q$ and $R$ are equal to $Q^*$ and $R^*$ 
by property (vii) of the functional calculus, an argument analogous 
to the one employed for \eqref{item:comp} shows that $\im{Q^*}+\im{R^*} \subset 
\dom{\overline{D_E}} $.  Hence we altogether have  
$\dom{D_E^*}\subset \dom{\overline{D_E}}$.}

\smallskip
It remains to prove the regularity of $\overline{D_E}$, i.e.~we have to show
that its graph is complemented. Write $\mcH$ for the Hilbert module of
$L^2$-sections of $S\tensor E$. By \eqref{item:comp} and \eqref{item:dom},
\begin{equation*}
  G(\overline{D_E}) = \{ (\overline{Q}x +\overline{T}y,
  \overline{D_EQ}x+\overline{D_ET}y)\mid (x,y)\in \mcH\times\mcH\} = U(\mcH\oplus\mcH)
\end{equation*}
where $U\colon\mcH\oplus\mcH\to\mcH\oplus\mcH$ is the bounded adjointable operator with $U(x,y) = (\overline{Q}x+\overline{T}y,
\overline{D_EQ}x+\overline{D_ET}y)$, using that $Q,T,D_EQ,D_E T$ are all
operators of non-positive order. As a graph of a closure, $U(\mcH\oplus\mcH)$ is
closed. By \cite[Theorem 3.2]{MR1325694}, the image of $U$ and hence the graph of $\overline{D_E}$ has an orthogonal complement.

\smallskip

Now we treat the general case where $M$ is complete, but not compact. We
will reduce this case to the compact one.
For this we use Theorem~\ref{thm:Characterization of self-adjoint, regular operators} to show that
$\overline{D_E}$ is self-adjoint and regular. The argument is inspired by the proof of
essential self-adjointness of the untwisted Dirac operator on complete
manifolds in \cite[Proposition 1.3.5]{MR2509837} and by the treatment of
\cite[Lemma 2.1]{MR2670972}. We will use three basic and well-known features:
\begin{enumerate}[(i)]
\item Given a compact submanifold $K\subset M$ with boundary, there exists a
closed Riemannian manifold $(M',g')$ equipped with a finitely generated
projective Hilbert $A$-module bundle $(E',\nabla^{E'})$ with metric connection,
a Dirac bundle $(S',\nabla^{S'})$, a {submanifold} $K'\subset M'$ and a
diffeomorphism $\psi\colon K\to K'$ such that $\psi$ preserves all the
structure (restricted to $K$ and $K'$, respectively). Specifically, $\psi^*g'|_{K'}=g_{|K}$
and $(\psi^*E',\psi^*\nabla^{E'})\cong (E,\nabla^E)_{|K}$ as bundles with
connections and $(\psi^*S',\psi^*\nabla^{S'})\cong (S,\nabla^S)_{|K}$ as Dirac
bundles.  
%Take M' as the double of K with a compact cylinder attached fi the boundary is non-empty and take the identity else.

\item\label{item:cutoff} Because $M$ is complete, for each compact subset $K\subset M$ and each
$\varepsilon>0$ there is a smooth function $\phi\colon M\to [0,1]$ with compact support such that
$\phi_{|K}=1$ and such that $\norm{\grad{\phi}}_{\infty}\leq\varepsilon$.

\item \label{item:Dphicomm}For each smooth function $\phi\colon M\to\mbR$ with compact support,
the commutator of multiplication by $\phi$ and $D_E$ extends to a bounded
operator on $L^2(M, S\otimes E)$ with norm bounded by $\norm{\grad{\phi}}_{\infty}$. 
More precisely, the commutator is given by Clifford multiplication with $\grad\phi$.
\end{enumerate}

Let $s\in \Gamma^{\infty}_{\mathrm{cpt}}(M, S\otimes E)$. For given $\varepsilon>0$ choose a function
$\phi\colon M\to [0,1]$ with compact support which is identically equal to $1$
on the support of $s$ and with $\norm{\grad{\phi}}_{\infty}\leq\varepsilon$. Then choose
a compact Riemannian manifold $(M',g')$ and bundles $(E',\nabla^{E'})$ as well as $(S',\nabla^{S'})$ with an isometry $\psi\colon K\pf K'$
where $K$ is a compact manifold with boundary containing the $1$-neighborhood
of the support of $\phi$.

In the sequel functions and sections with support in $K\subset M$ or
the corresponding set $K'$ in $M'$ will be transported back and forth using this
isometry without further comment. For example, we interchangeably think of $\phi$ as a
function on $M'$ and $s$ as a section of  $S'\otimes E'|_{M'}$.

Because $M'$ is compact the closure of the twisted Dirac operator $D_{E'}$ on $M'$ acting on
sections of $S'\otimes E'$ is already shown to be 
regular and self-adjoint. Hence by Theorem~\ref{thm:Characterization of self-adjoint, regular operators} we can find an 
element $x\in\dom{\overline{D_{E'}}}$ such that $(\overline{D_{E'}}+i)x=s$. We obtain
\[
   \langle s , s \rangle =\innerprod{(\overline{D_E}+i)x,(\overline{D_E}+i)x}=\innerprod{\overline{D_E} x,
\overline{D_E}x} +
\innerprod{x,x} \ge \langle x, x \rangle \in A_{+}.
\]
 Now $\phi\, x$ is defined on $M$ and belongs
to $\dom{\overline{D_E}}$ by \eqref{item:Dphicomm}. Moreover, 
\begin{equation*}
(\overline{D_E}+i)(\phi\, x)= (\overline{D_{E'}}+i)(\phi x) =[\overline{D_{E'}},\phi] x + \phi\,(
\overline{D_{E'}}+i) x
\text{ .} 
\end{equation*}
Here, now $\phi\,(\overline{D_{E'}}+i)x = \phi\,s=s$. On the other hand,
$\norm{[\overline{D_{E'}},\phi]}\leq\norm{\grad{\phi}}_{\infty}\leq\varepsilon$ so that
$\norm{[\overline{D_{E'}},\phi] x} \leq \varepsilon \norm{x}\leq\varepsilon\norm{s}$.

It follows that $s$ lies in the closure of the image of $\overline{D_E} + i$ and
therefore 
that $\overline{D_E} + i$ has dense range as $\Gamma^{\infty}_{\mathrm{cpt}}(M,
S\otimes E)$ is dense in $L^2(M, S\otimes E)$. {In the same 
way it is shown that $\overline{D_E} - i$ has dense range}. This implies the theorem by
Theorem~\ref{thm:Characterization of self-adjoint, regular operators}.
\end{proof}

\section{Positive scalar curvature, partial vanishing, and coarse index}

We now introduce the coarse index of the Dirac operator $D_E$ on a complete spin
manifold $(M,g)$, twisted by a smooth Hilbert $A$-module bundle $E$, and
prove the 
vanishing result Theorem \ref{thm:Vanishing_theorem_intro}. For simplicity we 
will use the notation $D_E$ for the densely defined, self-adjoint and regular 
closure $\overline{D_E}$ of $D_E$, see Theorem \ref{thm:Essential self-adjointness of twisted Dirac operator}.  

\subsection{The coarse index}
\label{subsect:The Roe indicies}

The construction of the index is based on  the functional calculus for regular 
and self-adjoint operators on Hilbert $A$-modules from \cite[Chapter 9 and 10]{MR1325694} 
and \cite[Section 3]{MR1949157}.
We will first recall this functional calculus in a form needed for our purpose.

\begin{thm}[Continuous functional calculus] \label{thm:Continuous functional calculus}
  Let $C(\mbR)$ be the $*$-algebra of continuous complex 
valued functions on $\mbR$. 
Let $T$ be a (possibly unbounded) regular, self adjoint operator on 
the Hilbert $A$-module $\mcH$. Then there is a $*$-preserving linear map 
\[
\pi_{T}\pp C(\mbR)\pf\mathcal{R}_A(\mcH)\,,\,f\mapsto f(T)
\]
with values in the set of regular operators on $\mcH$,  which has the following properties. 
\begin{itemize} 
  \item $\pi_T$ restricts to a $C^*$-algebra homomorphism
$\pi_{T}\pp C_b(\mbR)\pf\mathcal{L}_A(\mcH)$
on the set $C_b(\mbR)$ of bounded complex valued functions on $\mbR$. 
  \item If $|f| \leq |g|$, then $\dom{g(T)} \subset \dom{f(T)}$. 
   \item (Strong continuity) If $(f_n)_{n \in \mbN}$ is a sequence in $C(\mbR)$ which is 
  dominated by $F \in C(\mbR)$, i.e. $|f_n| \leq |F|$ for all $n$, and if $f_n \pf f$
uniformly on compact subsets of $\mbR$, then $\pi_T(f_n) x\pf \pi_T(f) x$ for each $x\in\dom{F(T)}$. 
   \item $\pi_T({\rm Id}) = T$.
     \item If $f \in C_b(\mbR)$ and $F \in C(\mbR)$ is defined by $F(t) = t \cdot f(t)$, then 
  $\dom{T} \subset   \dom{F(T)}$ and for all $x \in \dom{T}$ we have $F(T)x = Tf(T) x = f(T) T x$. 
  If $F$ is bounded, then $\im{f(T)} \subset \dom{T}$, and we have $F(T) = Tf(T)$ as 
  bounded operators on $\mcH$. 
\end{itemize} 

\end{thm}

Let $(M,g)$ be a complete Riemannian spin manifold and $(E,\nabla)$ a
smooth Hilbert $A$-module bundle on $M$ with metric connection and with
finitely generated projective fibers.
We will now define the coarse index $\ind(D_E)$, following {\cite{Roe_AMS}. It is an element of $ K_*( 
C^*(M;A))$, the $K$-theory of the  coarse $C^*$-algebra  $C^*(M;A)$ of $M$ with
coefficients in $A$. This $C^*$-algebra  was introduced in \cite{MR1451755} and its definition will 
be recalled shortly. 
The definition of $\ind(D_E)$ 
 uses 
the functional calculus for self-adjoint (unbounded) Hilbert $A$-module
operators in  Theorem \ref{thm:Continuous functional
  calculus}.}

We will work with the Hilbert $A$-module $\mcH:=L^2(M,S\tensor E)$, on which
$C_0(M)$, the $C^*$-algebra of all complex valued continuous functions
  on $M$ vanishing at infinity, acts by pointwise multiplication. The
corresponding representation is denoted 
$\rho\colon C_0(M)\to \mathcal{L}_A(\mcH)$. The following definition generalizes the 
corresponding notions from \cite[Chapter 3]{Roe_AMS} to the Hilbert $A$-module $\mcH$. 

\begin{definition} Let $T \in \mathcal{L}_A(\mcH)$.
\begin{itemize} 
   \item $T$ is \emph{locally compact} if
  $T\circ\rho(f)$ and $\rho(f)\circ T$ are $A$-compact operators for all $f\in
  C_0(M)$.
  \item $T$ is called \emph{pseudolocal} if the commutator  $[T,\rho(f)]$ is
  $A$-compact for any $f\in C_0(M)$. 
  \item $T$ has \emph{finite
    propagation} if there exists $R>0$ such that $\rho(f)\circ T\circ\rho(g)$
  vanishes for all $f,g\in C_0(M)$ with $d(\supp{f},\supp{g})\geq R$.  In this case 
  we say that $T$ has {\em propagation bounded by $R$}.
  \item The
  \emph{Roe $C^*$-algebra} associated with $\rho$ is the sub-$C^*$-algebra of
  $\mcL_A^*(\mcH)$ generated by all locally compact operators with finite
  propagation. It  will be denoted by $ C^*(M; A)$.
  \item  If $X\subset M$ is closed, we define $C^*(X\subset M;A)$ as  the closed
  ideal of $C^*(M;A)$ generated by locally compact operators $T$ of finite
  propagation which are \emph{supported near $X$}, i.e.~such that there is
  $R>0$ with $T\rho(f)=0$ and $\rho(f) T=0$ for all $f\in C_0(M)$ with
  $d(\supp{f},X)\geq R$.
\end{itemize}
\end{definition}

\begin{remark}\label{rem:functoriality}
  We suppress the dependence on the bundle $S\tensor E$ in the notation
  $C^*(M;A)$. This is justified by the following functoriality results
  \cite[Lemma 5.4, Proposition 5.5]{MR1451755}.

  For any Lipschitz map $f\colon M\to N$ and Hilbert $A$-module bundles $E\to
  M$, $F\to N$ so that the fibers of $F \to N$ are large enough  - adding a trivial
  bundle will always suffice if $M$ has positive dimension - there are
  canonical $C^*$-algebra homomorphisms
  $f_*\colon C^*(M;A)\to C^*(N;A)$, {obtained by  conjugation with an isometry
  between the Hilbert $A$-modules of sections of these bundles. 
  The induced map on K-theory is functorial in $f$.} For $f=\id_M\colon M\to M$ we can arrange that $f_*$ is an isomorphism, if the fibers of $E \to M$ and $F \to N$ are large enough. 

 % Here, if $M$ has positive dimension and if all fibers of $E$ contain a free
 % $A$-module then $E$ has large enough fibers. Therefore, one can achieve this
%  always by replacing $E$ with $E\oplus M\times A$. Throughout, we make the
 % assumption that our bundle $E$ is large enough in this sense.
\end{remark}

\begin{proposition}\label{prop:wave}
  Using the functional calculus of Subsection \ref{thm:Continuous functional
    calculus}  we define the wave operator group
  $\{\exp(isD_E)\}_{s\in\mbR}$ which consists of unitary
  operators. It satisfies
  the wave equation: for $u\in\dom{D_E}$,
  \begin{equation*}
    \frac{d}{ds} \exp(is D_E) u = iD_E \exp(is D_E) u.
  \end{equation*}
Moreover, each $\exp(isD_E)$ is a finite
  propagation operator with propagation $\abs{s}$.
\end{proposition}

\begin{proof} Because the function $t \pf \exp(ist)$ is bounded on $\mbR$, the operators $\exp(isD_E)$ 
are bounded adjointable and unitary by the properties of the functional
calculus Theorem \ref{thm:Continuous functional calculus}.

   For fixed $s\in\reals$,  $\big( \exp(i(s+h)t) -\exp(ist) \big) /h$ converges to $i{t}
   \exp(ist)$ uniformly for $t\in [-R,R]$ for each $R$ and with a uniform
   bound of the difference quotients by $\abs{1+t}$. The claim about the wave
   equation then follows from   the strong continuity property in Theorem \ref{thm:Continuous
     functional calculus}. 

   The unit propagation property is a standard fact which follows from a
   priori energy estimates. The proof given in \cite[Proposition
   10.3.1]{MR1817560} only uses properties of the wave equation, some elementary 
   properties of the functional calculus and the fact that, for a smooth
   function $g\colon M\to \reals$, the commutator $[D_E,\rho(g)]$ is equal to 
   Clifford multiplication with the gradient of $g$. It therefore 
   generalizes     immediately  from the case of  operators on
   Hilbert spaces treated in \cite[Proposition 10.3.1]{MR1817560} to the
   unbounded operator $D_E$ on the Hilbert $A$-module $\mcH$.
\end{proof}

\begin{defi}
An odd function $\chi\in C(\mbR)$ is called \emph{normalizing function} if
$\chi(t)\pf\pm 1$ as $t\pf\pm\infty$.  
\end{defi}

\begin{lem}
\label{lem:Functional_calculus_via_wave_group} For the functional calculus of
the regular self-adjoint operator $D_E$ the following assertions hold. 
\begin{itemize} 
   \item[a)] 
For any $\varphi\in C^{\infty}(\mbR)\cap L^1(\mbR)$ with $\hat\varphi\in
C_{\mathrm{cpt}}^{\infty}(\mbR)$, one has
\begin{equation*}
\label{eq:Integral_representation}
\varphi(D_E)
u=\frac{1}{2\pi}\int_{\mbR}\hat{\varphi}(s)\,
\exp(isD_E)u\,ds 
\end{equation*}
for all compactly supported smooth sections
$u\in\Gamma_{\mathrm{cpt}}^{\infty}(M,S \otimes E)$. Further, the operator 
 $\varphi(D_E)$ is locally compact and of  finite propagation.
\item[b)] For arbitrary $\varphi \in C_0(M)$ we have $\varphi(D_E)\in
C^*(M;A)$.
\item[c)] If  $\chi\in C_b(\reals)$ is a normalizing function then $\chi(D_E)$ is a norm limit 
of bounded, self-adjoint finite propagation operators.
\end{itemize} 
\end{lem}

\begin{proof} For a) we first assume that $\supp{\hat\varphi}\subset [-R,R]$ for $R>0$.
By Theorem~\ref{thm:Continuous functional
  calculus}  the compactly supported integrand is 
continuous and the  integral is defined as a limit of Riemann
sums. Moreover we have  $\varphi(t) = \frac{1}{2\pi}\int_\reals
\hat\varphi(s)\exp(ist)\;ds$, again as a limit of Riemann sums (and uniformly
for $t$ in compact subsets of $\reals$), by the Fourier inversion
theorem. The equation in a) now follows from the continuity statement in 
Theorem \ref{thm:Continuous functional calculus}.

That $\varphi(D)$ has finite propagation is an immediate consequence of the
integral 
representation of this operator, since the wave operators
$\exp(isD_E)$ have finite propagation $\abs{s}$. 

To prove local compactness of $\varphi(D_E)$, let $f$ be a compactly supported
smooth function on $M$. Note that the propagation of $\varphi(D_E)$ is
bounded by $R$. Let $g$ be a compactly supported smooth function which is
identically equal to $1$ on the $R$-neighborhood of the support of $f$. Then
$\varphi(D_E)\rho(f) =\rho(g)\varphi(D_E)\rho(f)$ and 
$\rho(f)\varphi(D_E)=\rho(f)\varphi(D_E)\rho(g)$.

Next, as in Subsection \ref{sec:regularity}, when reducing from complete to
compact manifolds, we can find an isometry of a suitable neighborhood of
$\supp{g}$ with target a suitable subset of a compact manifold $M_1$, covered
by an isometry of $E$ to a Hilbert $A$-module bundle $E_1$ on $M_1$, when
both bundles are restricted to the respective subsets of $M$ and $M_1$. 

{This induces an isometry which conjugates $\rho(g) \varphi(D_E)\varphi(f)$ to
the corresponding operator $\rho(g_1)\varphi(D_{E_1})\rho(f_1)$ on $M_1$. This 
assertion uses the integral representation and the fact that the family $\exp(isD_E)
u$ on $M$ is conjugated to the corresponding family $\exp(is D_{E_1}) u_1$ on 
$M_1$ as long as $\supp{u}\subset \supp{f}$ and $\abs{s}\le R$. Here we observe 
that the latter
 is the unique  solution of the wave equation for a 
given initial function $u$, which follows immediately from the a priori energy
estimates for 
the wave operator mentioned in the proof of Proposition \ref{prop:wave}.}

Now we use the parametrix $Q_1$ for $D_{E_1}$ of order $-1$ with remainder
$R_1$ of order $-\infty$ such that $D_{E_1}Q_1 = 1-R_1$. {Composing with 
the bounded operator  $\varphi(D_E)$ from the left leads to the equation} 
\begin{equation*}
  \varphi(D_{E_1}) = (\varphi(D_{E_1})D_{E_1}) \circ Q_1 + \varphi(D_{E_1})
  \circ R_1. 
\end{equation*}
Here $\varphi(D_{E_1})D_{E_1}$ {which is defined a priori 
only on $\dom{D_{E_1}}$}, can be extended to a bounded operator on $\mcH$, as
$t \mapsto t \varphi(t)$ is bounded, and $Q_1,R_1$ are $A$-compact because they
are of negative order in the
pseudodifferential calculus on $M_1$. Consequently, since
the $A$-compact operators are an ideal in the bounded operators also
$\rho(g_1)\varphi(D_{E_1}) \rho(f_1)$ and its conjugate
$\rho(g)\varphi(D_E)\rho(f)$ are $A$-compact. The same argument implies
that $\rho(f)\varphi(D_E)\rho(g)$ is $A$-compact.

The claim about arbitrary $\varphi\in C_0(M)$ follows from the usual
{density} argument.

{We now prove c).} The function $f(x)=\frac{x}{\sqrt{1+x^2}}$ is a normalizing
function, any 
other such function $\chi$ satisfies $\chi-f\in
C_0(\reals)$. {Because of b) it} suffices to prove the statement for $f$. Now,
$f(x)=xg(x)$ with $g(x)=(1+x^2)^{-1/2}$. We construct a sequence of {bounded} continuous functions $g_n$ such that the functions $x \mapsto x g_n(x)$ are also bounded and
\begin{enumerate}[(i)]
\item $\lim_{n\to \infty}\norm{xg_n(x)-xg(x)}_\infty =0$. 
\item $g_n$ has a smooth Fourier transform with compact support.
\end{enumerate}

Then the sequence of bounded operators 
 $D_E g_n(D_E)$ converges by the functional calculus in norm to $D_E
 g(D_E)=f(D_E)$. Moreover, $g_n(D_E)$ has finite propagation exactly by
the same Fourier inversion argument which showed that $\varphi(D_E)$ has finite
propagation. As $D_E$ itself has propagation $0$, the composition $D_E
g_n(D_E)$ also has finite propagation. {Hence assertion c) holds 
with $\chi$ replaced by $f$.} 

To construct $g_n$, consider first the Fourier transform $\hat{g}(\xi)$. This
is, up to a constant, the modified Bessel function $K_0(\abs{\xi})$ of the
second kind \cite[p.~376]{Abramowitz}. The following Lemma
\ref{lem:FT} shows that this function is {square integrable},
smooth outside $0$ and {of Schwartz type as
$\xi\to\pm\infty$, meaning that} $\lim_{\abs{\xi}\to\infty}\abs{\xi^k 
  \frac{d^l\hat{g}}{d\xi^l}(\xi)}=0$ for all $k,l$. Choose smooth cutoff
functions $\phi_n\colon \reals\to [0,1]$
with {$\norm{\phi^{(k)}_n}_\infty \le 1$} for $k=0,1$ and such that 
$\phi_n(\xi)=1$ for $\abs{\xi}\le n$. Set $\hat{g}_n = \phi_n\hat g$ {and let $g_n$ be 
the Fourier transform of $\hat{g}_n$. Being equal to  the convolution of the $L^2$-function $g$ and 
the {Schwartz} function $\hat{\phi}_n$ the function $g_n$ is
bounded and continuous.
}
We obtain
\begin{equation*}
 \abs{x (g(x)-g_n(x))} \le \int_{\reals} \abs{(\hat
     g-\hat{g}_n)'(\xi)}\,d\xi \le \int_{\abs{\xi}\ge n}
 \abs{(\hat g(1-\phi_n))'(\xi)}
 \;d\xi\xrightarrow{n\to\infty} 0 
\end{equation*}
as $\hat g(\xi)$ is rapidly decreasing for $\abs{\xi}\to\infty$. {Furthermore, the last
inequality also shows that $x \mapsto x g_n$ is bounded for all $n$}. 
\end{proof}

\begin{lemma}\label{lem:FT}
  Set $g(x)=\frac{1}{\sqrt{1+x^2}}$. Then $g\in L^2(\reals)$ and its Fourier
  transform $\hat{g}$ has the following properties:
  \begin{enumerate}
  \item For each $k>0$ and each $0\le l<k$ the function $\frac{d^l}{d\xi^l}
    (\xi^k \hat g(\xi))$ belongs to $L^2(\reals)$.
  \item The restriction of $\hat{g}$ to $\reals\setminus\{0\}$ is smooth
  \item The restriction of $\xi^k \frac{d^l}{d\xi^l}
    \hat g$ to $\reals\setminus (-1,1)$ is bounded for each
    $k,l\in\naturals$. 
  \end{enumerate}

  % Consequently, using suitable cutoff functions $\phi_n$ and then choosing
  % $g_n$ with $\hat{g_n} =\phi_n \hat g$ we can achieve that $\hat{g_n}$ is
  % compactly supported for all $n$,$\lim_{n\to\infty}
  % \norm{g_n-g}_\infty =0$ and $\lim_{n\to\infty}\norm{xg_n(x)-xg(x)}_\infty =0$.
\end{lemma}
\begin{proof}
  An explicit calculation shows that $x^l\frac{d^k}{dx^k} g(x)$ belongs to
  $L^2(\reals)$ for $l<k$ and to $L^1(\reals)$ for $l<k+1$.

  For the Fourier transforms, we therefore get that
  $\frac{d^l}{d\xi^l}\left(\xi^k \hat g\right)$ belongs to $L^2(\reals)$ for
  $l<k$ and to $L^\infty(\reals)$ for $l<k+1$.  
  The Sobolev embedding theorem then implies that $\xi^k\hat g(\xi)$ belongs
  to $C^{k-1}(\reals)$. As $\xi^k$ is smooth and invertible outside the
  origin, this implies that $\hat g$ is smooth on
  $\reals\setminus\{0\}$. Calculating $\frac{d^l}{d\xi^l}(\xi^k\hat g)$ with
  the product rule, by  induction on $l$ we establish that the restriction of
  $\xi^k\frac{d^l}{d\xi^l}\hat  g$ to $\reals\setminus (-1,1)$ is bounded for
  each $k,l$.
\end{proof}
  % Choose smooth compactly supported cutoff functions $\phi_n\colon\reals\to
  % [0,1]$ which are  identically $1$ on $[-n,n]$ and such that
  % $\norm{\phi_n'}_\infty\le   1$. Then $\norm{(1-\phi_n)\hat
  %   g}_\infty\xrightarrow{n\to \infty} 0$. and also $\norm{\hat g' -
  %   (\phi_n\hat g)'}_{L^1} \le \norm{(1-\phi_n)\hat g'}_{L^1} +
  % \norm{\phi_n' \hat g}_{L^1} \xrightarrow{n\to\infty} 0$, which implies the
  % desired statement for the Fourier transform. 

Let $\mcM$ be the sub-$C^*$-algebra of $\mcL_A(\mcH)$
generated by all operators with finite propagation. Then
$C^*(M;A)\subset\mcM$ is a $C^*$-ideal (for the ideal property 
use an argument similar to the third paragraph in the proof of 
Lemma \ref{lem:Functional_calculus_via_wave_group}). 
Consider the associated six-term exact sequence
\begin{equation}
\label{eq:six_term_exact_sequence}
\begin{CD}
 K_0(C^*(M;A)) @>>>  K_0(\mcM) @>>>  K_0(\mcM/C^*(M;A)) \\
@A\partial_1 AA @. @VV\partial_0 V \\
 K_1(\mcM/C^*(M;A))  @<<<  K_1(\mcM) @<<<  K_1(C^*(M;A))  %>>>>
\end{CD}
\text{ .}
\end{equation}

\begin{definition}
  If $\dim(M)$ is odd, $\frac{1+\chi(D_E)}{2}$ belongs to $\mcM$ and is a projection
  modulo $C^*(M;A)$, as $\chi(D_E)^2-I \in C^*(M;A)$. The \emph{coarse index}
  $\ind(D_{E})$ is then defined as
  \begin{equation*}
    \ind(D_E):=\partial_0[\frac{1}{2}(1+\chi(D_E))]\in K_1( C^*(M;A)).
  \end{equation*}

  \medskip If $\dim(M)$ is even, the decomposition of the spinor  bundle $S \to M$ in
  even and odd parts induces a decomposition $L^2(M,S \otimes
  E)=\mcH_0\oplus\mcH_1$ for which
  \[
  D_{E}=
  \begin{bmatrix}
    0       & D_1   \\
    D_0 & 0
  \end{bmatrix}
  \text{ .}
  \]

  Then $U^*\chi(D_E)_0$ belongs to $\mcM$, where $U\pp\mcH_0\pf\mcH_1$ is a
  unitary embedding\footnote{See \cite[p.~91 f]{MR1219916} for the definition
    and a proof for the existence of such an isometry. This notion is used for
    the functoriality of Remark \ref{rem:functoriality}}. $\id_M$, and
  $U^*\chi(D_E)_0$ is unitary modulo $C^*(M;A)$ as $\chi^2(D_E)-I\in
  C^*(M;A)$, i.e.~$U^*\chi(D_E)_0$ represents an element in
  $K_1(\mcM/C^*(M;A)) $.

  Then the \emph{coarse index} $\ind(D_E)$ is defined as
  \begin{equation*}
    \ind(D_E):=\partial_1[U^*\chi(D_E)_0]\in K_0( C^*(M;A)).
  \end{equation*}
\end{definition}

% \begin{remark}  In the case that $A=\complexs$, one often uses the (smaller)
%   $C^*$-algebra $D^*(M)$, generated by all finite propagation operators
%   which are pseudolocal, instead of $\mcM$ ---compare e.g.~\cite[Section
%   5]{Roe_AMS}. This is important because $D^*(M)/C^*(M)$ computes the
%   K-homology of $M$.
%   In this spirit, one might be tempted to use an algebra
%  ``$D^*(M;A)$'' (generated by the finite propagation pseudolocal operators) also for
%  general $A$ instead of $\mcM$. It is not hard to show that $\chi(D_E)$ 
%   belongs to this algebra (but not important for our purposes).
 
%   In any case, the notation $D^*(M;A)$ would not be appropriate. It suggests 
%   that $D^*(M;A)/C^*(M;A)$ is of homological nature (K-homology of $M$
%   with coefficients in $A$), which in general is not the case. An appropriate
%   definition of $D^*(M;A)$ from this point of view would have to incorporate further
%   conditions on the operators, similar to the ``liftability'' condition in
%   \cite[Definition 5.11]{HR4}. 
% \end{remark}

%(This definition would of course also make sense if $\dim(M)$ were even but one would
%have $\ind(D_{M,\mcV(M)})=0$ anyhow.)

\subsection{The vanishing theorem}

The following vanishing theorem generalizes an analogous result from
\cite[Proposition 3.11 and the following Remark]{Roe_AMS}, \cite{MR1147350} and
\cite[Proposition 4]{MR1435703}
for the spin Dirac operator to the case of the spin Dirac operator
twisted with a flat Hilbert $A$-module bundle.

\begin{definition}
  Let $M$ be a complete Riemannian manifold and $A$ a unital $C^*$-algebra. A 
  closed subset
  $X\subset M$ is called \emph{coarsely $A$-negligible} if the inclusion
  induces the zero homomorphism $0=i_*\colon K_*(C^*(X\subset M;A))\to
  K_*(C^*(M;A))$.

  Note that, by functoriality, every closed subset of a coarsely $A$-negligible set
  is itself coarsely $A$-negligible. Secondly, note that, by definition,
  $C^*(U_R(X)\subset M;A)=C^*(X\subset M;A)$ for any closed $R$-neighborhood
  $U_R(X)$ of $X$. 
\end{definition}

\begin{proposition}
  If $M$ is  a complete connected non-compact Riemannian manifold then every
  compact subset $K\subset M$ is coarsely $A$-negligible for any  $A$. 
\end{proposition}
\begin{proof}
  Choose an
isometric embedding $\gamma\pp\mbR_+\pf M$. This is possible because $M$ is
complete, connected and
non-compact, see~\cite[p.~92]{MR2088027}. Then, because a compact set has
finite diameter, $K\subset U_R(\gamma(\reals_+))$ for $R$ sufficiently large.

It remains to show that $K_*(C^*(\gamma(\reals_+  )\subset
M;A))=0$, so that $\gamma(\reals_+)$ is $A$-negligible. This follows from 
$K_*(C^*(\gamma(\reals_+  )\subset
M;A)) \cong K_*(\reals_+;A)$ by \cite[Proposition 2.9]{SchickZadeh} and 
$K_*(\reals_+;A)=0$ by an Eilenberg swindle argument as carried out  in
\cite[Proposition 9.4]{Roe_AMS}. Compare \cite[Proposition 2.6]{SchickZadeh}
for the generalization to Hilbert $A$-module coefficients. 
\end{proof}

\begin{thm}
\label{thm:vanishing_theorem}
Let $(M,g)$ be a complete Riemannian spin manifold with
uniformly positive scalar curvature
outside of an $A$-negligible set $X$. {Let $E\to M$ be a smooth finitely generated 
projective Hilbert  $A$-module bundle equipped with a flat metric connection.} 
Then the coarse index $\ind(D_E)\in K_*(C^*(M;A))$ of the
twisted Dirac operator of $(M,g)$ vanishes.

In particular, if $M$ is non-compact connected and has uniformly positive
scalar curvature outside a compact set, then $\ind(D_E)=0$ for any flat
Hilbert $A$-module bundle $E$ as above. 
\end{thm}
\begin{proof}
We use the notation from diagram~(\ref{eq:six_term_exact_sequence})
and consider the following commutative diagram:
\begin{equation}
\begin{CD}
0 @>>>C^*(M;A) @>>> \mcM @>>> \mcM/C^*(M;A) @>>> 0 \\
  @. @A{i_*}AA @| @AAA \\
0 @>>> C^*(X\subset M;A) @>>> \mcM @>>> \mcM/C^*(X\subset M;A) @>>> 0 \\  
\end{CD}
\end{equation}
In Proposition \ref{prop:lift} below we will construct a normalizing function $\chi$ for which $\chi(D_E)^2$ equals $I$ modulo 
$C^*(X\subset M;A)$. This implies that  $[\frac{1}{2}(1+\chi(D_E))]$ lifts to
$K_0(C^*(X\subset M;A))$ (if $\dim M$ is odd) and 
$[U^*\chi(D_E)_0]$ lifts to
$K_1(C^*(X\subset M;A))$ (if $\dim M$ is even). 

Using the 
commutativity of
\begin{equation*}
  \begin{CD}
      K_{*+1}(\mcM/ C^*(X\subset M;A))  @>{\partial_{*+1}}>> K_*(C^*(X\subset
      M;A))\\
     @VV{i_*}V  @VV{i_*\mathbf{=0}}V\\
   K_{*+1}(\mcM/C^*(M;A))  @>{\partial_{*+1}}>> K_*(C^*(M;A))
  \end{CD}
\end{equation*}
in the six-term exact sequence, 
we obtain  the desired
result by the $A$-negligibility of $X$. 
\end{proof}

To prove that $I-\chi(D_E)^2\in C^*(X\subset M;A)$ for a suitable $\chi$, we
use the following criterion.

\begin{lemma}\label{lem:rel_crit}
  An operator $T\in C^*(M;A)$ belongs to $C^*(X\subset M;A)$ for a closed
  subset $X$ if and only if for each $\epsilon>0$ there is $R>0$ such that for
  each $u\in L^2(M,  S \otimes E)$ with support outside the $R$-neighborhood $U_R(X)$ we
  have
  \begin{equation}\label{eq:aboveineq}
\norm{Tu}\le \epsilon\,\norm{u}.
\end{equation}

\end{lemma}
\begin{proof}
  If $T$ is a norm limit of operators which are supported in $R$-neighborhoods
  of $X$, the inequality \eqref{eq:aboveineq} obviously holds.

  Conversely, write $T=\lim T_n$ with operators $T_n$ which are locally
  compact and of finite propagation. Using the estimate, we have to modify
  the the operators $T_n$ such that they are in addition supported in a
  bounded neighborhood of $X$. For this we  choose cutoff functions
  $\phi_n\colon M\to [0,1]$ which are supported in $U_{2n}(X)$ and which are
  identically equal 
  to $1$ in $U_n(X)$. Then $T_n\rho(\phi_n)$ are still locally compact, of
  {some} finite propagation $P_n$, and,
   in addition, supported in $U_{2n+P_n}(X)$, i.e. $(T_n \rho(\phi_n))  \rho(f) = 0$,
   {if 
   $\supp{f} \cap U_{2n+P_n}(X) = \emptyset$}. 

We only have to show that $T_n$ and $T_n\rho(\phi_n)$ are close in operator
norm if $n$ is sufficiently large. For $\epsilon>0$ choose $R$ as in the
assumption and $n>R$. Then 
\[
   \norm{(T_n-T_n\rho(\phi_n))u} = \norm{T_n(1-\rho(\phi_n)) u} \le \epsilon \norm{(1-\rho(\phi_n))u} \le \epsilon\norm{u},
 \]
 for each $u \in L^2(M, S\otimes E)$, as $(1-\rho(\phi_n))u$ has support outside the
$R$-neighborhood of $X$. Therefore $\norm{T_n-T_n\rho(\phi_n)}\le \epsilon$.
\end{proof}

Before we can prove the required Proposition \ref{prop:lift}, we need two
further preparatory lemmas. {The first is a standard property
  of the Fourier 
transform, its proof is left to the reader.}

%Let $\mathcal{W}$ be the set of all smooth $L^1$-functions with compactly supported
%Fourier transform. We leave the proof of the following result to the reader.

\begin{lem}
\label{lem:lemma_one_from_appendix}
Let $f\in C_{\mathrm{cpt}}^{\infty}(\mbR)$. {Then for each $\delta>0$ there exists 
a smooth $L^1$-function $f_{\delta}$ with compactly supported Fourier transform 
 and such that for all
 $x\in\mbR$ and for $j=0,1,2$ we have $\abs{x^j\,(f(x)-f_{\delta}(x))}\le\delta$.}
 \end{lem}

\begin{lem}
\label{lem:lemma_two_from_appendix}
Let $f\in C_{\mathrm{cpt}}^{\infty}(\mbR)$ with $f\geq 0$. Then  for each
$\varepsilon>0$ there exists a decomposition $f=f_{\varepsilon}+g_{\varepsilon}$ and
$S(\varepsilon)>0$  with the following properties:
\begin{enumerate}[(i)]
\item $f_{\varepsilon}=F_{\varepsilon}^2$ with $F_{\varepsilon}$ 
  a smooth $L^1$-function and $\supp{\hat{F}_{\varepsilon}}\subset [-S(\varepsilon),S(\varepsilon)]$.
\item $\sup\{\abs{x^j\,g_{\varepsilon}(x)}\,;\,x\in\mbR\}\leq\varepsilon$ for
  each $j=0,1,2$.
\end{enumerate}
\end{lem}

\begin{proof}
Set $F:=f^{1/2}$, choose $A>0$ with $\supp{F}\subset [-A,A]$. Let
$\varepsilon>0$. 
Approximate $F$ by $H\in C^{\infty}(\mbR)$ with
$\norm{F-H}\leq\frac{\varepsilon}{(A+1)^2 \cdot 4(\norm{F}+1)}$, 
$\supp{H}\subset [-A-1,A+1]$ and $\norm{H}\le \norm{F}+1$. Here and in 
the remainder of the proof we use the maximum norm on $C_{\mathrm{cpt}}^{\infty}({\mbR})$. 

 Choose
$\delta\le \frac{\varepsilon}{4 (\norm{F} +1)}$, $\delta\le 1$. By 
Lemma~\ref{lem:lemma_one_from_appendix}, $H$ admits a decomposition
$H=H_{\delta}+R_{\delta}$ where $H_{\delta}$ is a smooth $L^1$-function with
compactly supported Fourier transform, such that
$\sup\abs{x^jR_{\delta}(x)}\leq\delta$ for $j=0,1,2$, and
$\norm{H_{\delta}}\leq\norm{H}$ and with $\supp{\hat{H}_{\delta}}\subset
[-R(\delta),R(\delta)]$ {for suitable $R(\delta)>0$}. The
estimate on $R_\delta$ implies
$\norm{H_\delta}\le \norm{H}+\delta$.

 Set $f_{\varepsilon}:=H_{\delta}^2$ and
$F_{\varepsilon}:=H_{\delta}$. Then (i) holds with
 $S(\varepsilon):=R(\delta)$.
Finally, we obtain (ii) from the following estimate for $j=0,1,2$
\begin{equation*}
\belowdisplayskip=-15pt
\begin{aligned}
\abs{x^j(f(x)-f_{\varepsilon}(x))}
&=
\abs{x^j(F(x)-F_{\varepsilon}(x))\,(F(x)+F_{\varepsilon}(x))} \\
&\le 
\abs{x^j(F(x)-H_{\delta}(x))}\,\norm{F+H_{\delta}} \\
&\leq
\left(\abs{x^j(F(x)-H(x))}+\abs{x^j(H(x)-H_{\delta}(x))}\right)\,\norm{F+H_{\delta}} \\
&\leq
\left(\abs{x^j(F(x)-H(x))}+\abs{x^jR_{\delta}(x)}\right)\left(\norm{F}+\norm{H_\delta}\right) \\
&\leq
 \left( (A+1)^j\norm{F-H} +\delta \right) \left(\norm{F} +\norm{H}
   +\delta\right) \\
&\leq \left(\frac{\varepsilon}{4(\norm{F}+1)}
  +\frac{\varepsilon}{4(\norm{F}+1)} \right)
\left(\norm{F}+\norm{F}+1+1\right) = 
\varepsilon\text{ .}
\end{aligned}
\end{equation*}
\end{proof}

\begin{proposition}\label{prop:lift}
  In the situation of Theorem \ref{thm:vanishing_theorem}, there is a
  normalizing   function $\chi$ such that $\chi(D_E)^2-I\in C^*(X\subset M;A)$.
\end{proposition}

\begin{proof}
Let $s_0>0$ be such that $s:=\mathrm{scal}\geq
s_0$ outside of $X$, choose $0<R<(s_0/4)^{1/2}$ and let $\chi$ be such that
$\varphi:=1-\chi^2$ is compactly supported in $[-R,R]$. 

Let $\varepsilon>0$. We will derive the
inequality
\begin{equation}
\label{eq:very_impressive_estimate}
\norm{\varphi(D_E)u}^2\leq \left(\frac{s_0}{4}-R^2\right)^{-1}\left(\frac{s_0}{4}+1\right)\,\varepsilon\,\norm{u}^2
\end{equation}
for each $u\in\Gamma_{\mathrm{cpt}}^{\infty}(M,S)$ with $\supp{u}$ outside of
$B(K;3\,S(\varepsilon))$ for an $S(\varepsilon)>0$. Using
Lemma~\ref{lem:rel_crit} this implies that $\varphi(D_E)\in C^*(X\subset
M;A)$, as required. 

In order to obtain~(\ref{eq:very_impressive_estimate}) we use
Lemma~\ref{lem:lemma_two_from_appendix} and write
$f=\varphi^2=f_{\varepsilon}+g_{\varepsilon}$ with
\begin{enumerate}[(i)]
\item $f_{\varepsilon}=F_{\varepsilon}^2$ with $F_{\varepsilon}\in C^{\infty}(\mbR)\cap L^1(\mbR)$ and $\supp{\hat{F}_{\varepsilon}}
\subset [-S(\varepsilon),S(\varepsilon)]$,

\item $\sup\{\abs{x^j\,g_{\varepsilon}(x)}\,;\,x\in\mbR\}\leq\varepsilon$ for each $j=0,1,2$.
\end{enumerate}
From the functional calculus of 
Theorem~\ref{thm:Continuous functional calculus} and  the self-adjointness of $D_E$  we obtain the
following estimates for each $u\in\Gamma_{\mathrm{cpt}}^{\infty}(M,S\tensor E)$:  
\begin{align}
%\label{eq:Basic_estimates_I}
%\innerprod{\varphi(D_E)u,\varphi(D_E)u}  & =  \langle f(D_E)u,u\rangle         \\ 
\label{eq:Basic_estimates_II}
\langle D_E^2f(D_E) u,u\rangle  & = \langle D_E^2\varphi(D_E)^2 u,u\rangle
 \leq R^2\,
     \langle \varphi(D_E)u, \varphi(D_E) u \rangle  \\
\label{eq:Basic_estimates_III}
\langle D_E^2g_{\varepsilon}(D_E)u,u\rangle  &\leq  \varepsilon\,\norm{u}^2\cdot 1_A \\
\label{eq:Basic_estimates_IV}
\langle g_{\varepsilon}(D_E)u,u\rangle &\leq \varepsilon\,\norm{u}^2\cdot 1_A
\end{align}
The first inequality uses  $D_E\varphi(D_E) = m(D_E)\varphi(D_E)$
for a suitable function $m\colon \reals\to \reals$ with $\|m\|_\infty\le R$ as
$\supp{\varphi}\subset [-R,R]$ and 
 the second inequality is based on the estimate
\[
 \langle D_E^2g_{\varepsilon}(D_E)u,u\rangle \leq \| \langle D_E^2 g_{\varepsilon}(D_E) u, u \rangle \| \cdot 1_A \leq \| D_E^2 g_{\varepsilon}(D_E) \| \cdot \|u\|^2 \cdot 1_A 
\]
which follows from the inequality $0 \leq a \leq \|a\|_A \cdot 1_A$ for $a \in A_+$ and the Cauchy-Schwarz inequality
for Hilbert $A$-modules. The third inequality above is derived in the same manner.

Using that $D_E$ has unit propagation speed one sees that the inclusion
$\supp{\hat{F}_{\varepsilon}}\subset [-S(\varepsilon),S(\varepsilon)]$ from
(i) implies $\supp{F_{\varepsilon}(D_E)u}\subset
B(\supp{u}, S(\varepsilon))$. In particular, $\supp{F_{\varepsilon}(D_E)u}$ is
outside of $B(X, S(\varepsilon))$ if $\supp{u}$ is outside of
$B(X;3\,S(\varepsilon))$. 

From the Schr\"odinger-Lichnerowicz formula
\begin{equation*}
D_E^2=\Delta
+s/4\text{ ,}
\end{equation*}
where $\Delta$ is a positive operator and $s$ denotes multiplication with the 
scalar curvature function, we
obtain for such $u$: 
\begin{equation}
\label{eq:Basic estimate_IV}
\begin{aligned}
&\phantom{{}={}}
\langle D_E^2 f_{\varepsilon}(D_E)u,u\rangle \\
&=
\langle F_{\varepsilon}^2(D_E)D_E^2u,u\rangle \\
&=
\langle \left\{\Delta+\frac{s}{4}\right\}F_{\varepsilon}(D_E)u,F_{\varepsilon}(D_E)u\rangle\\
&=
\langle \left\{\Delta+\left(\frac{s-s_0}{4}\right)\right\}F_{\varepsilon}(D_E)u,F_{\varepsilon}(D_E)u\rangle+\frac{s_0}{4}\,\langle f_{\varepsilon}(D_E)u,u\rangle                              \\
&\geq
\frac{s_0}{4}\,\langle f_{\varepsilon}(D_E)u,u\rangle\text{ .}
\end{aligned}
\end{equation}
Here we have used that $\supp{F_{\varepsilon}(D_E)u}$ is outside of $X$ and
hence that $s\geq s_0$ holds there, hence $\Delta + (s-s_0)/4$ acts as a
positive operator on
$F_\epsilon(D_E)u$. Using~(\ref{eq:Basic_estimates_II})--(\ref{eq:Basic
  estimate_IV}) one finally obtains
(\ref{eq:very_impressive_estimate}) from the following 
inequality in $A$ 
\begin{align*}
&\phantom{{}\leq}
\left(\frac{s_0}{4}-R^2\right)\,\langle \varphi(D_E)u, \varphi(D_E)u \rangle\\
&\leq
\frac{s_0}{4}\,\langle f(D_E)u,u\rangle
-\langle D_E^2f(D_E)u,u\rangle
\\
&\leq
\frac{s_0}{4}\,\langle f(D_E)u,u\rangle
-\langle D_E^2f_{\varepsilon}(D_E)u,u\rangle
-\langle D_E^2g_{\varepsilon}(D_E)u,u\rangle                \\
&\leq 
\frac{s_0}{4}\,\langle f(D_E)u,u\rangle
-\langle D_E^2f_{\varepsilon}(D_E)u,u\rangle
+\varepsilon\,\norm{u}^2\cdot 1_A                                \\
&\leq 
\frac{s_0}{4}\,\langle f(D_E)u,u\rangle
-\frac{s_0}{4}\,\langle f_{\varepsilon}(D_E)u,u\rangle+ \varepsilon \|u\|^2\cdot 1_A  \\
&= 
\frac{s_0}{4}\,\langle g_{\varepsilon}(D_E)u,u\rangle
+\varepsilon\,\norm{u}^2\cdot 1_A\\
&\leq 
\varepsilon\,\left(\frac{s_0}{4}+1\right)\,\norm{u}^2\cdot 1_A 
\text{ ,}
\end{align*}\
which implies the required inequality in $\mbR_+$ after
applying the norm of $A$.
\end{proof} 

\section{Codimension two index obstruction to positive scalar curvature}

In \cite[Theorem 2.6]{MR2670972} Roe's partitioned manifold index theorem
\cite[Theorem 4.4]{Roe_AMS} was 
generalized to Dirac operators twisted with Hilbert $A$-module bundles. Since this version 
will be used in the proof of our main result, we will briefly restate it here. 

\begin{thm}
\label{thm:PMIT}
Let $M$ be an odd-dimensional complete spin manifold with $\dim(M)\geq
3$ and let $N\subset M$ be a closed submanifold of codimension one with
trivial normal bundle, which divides $M$ into two parts $M_0$ and $M_1$ with
common boundary $N$. 
Denote with  $D_E$ the spin Dirac
 operator twisted by the Hilbert $A$-module bundle $E \to M$.

Let 
$\varphi_N\pp K_1(C^*(M;A))\pf K_0(A)$
be the generalization {to Hilbert $A$-module bundles} of the
homomorphism defined by the partitioning hypersurface as in 
\cite[Section 4]{Roe_AMS}. Then
\[
\varphi_N(\ind(D_{M,E}))=\ind(D_{N,E_{|N}})
\]
where $\ind(D_{N,E|_N})\in K_0(A)$ is the classical Mishchenko-Fomenko index
of the Dirac operator on the compact manifold $N$ twisted by $E|_N$.
\end{thm}

Recall that on an arbitrary connected manifold $M$ with fundamental group $\pi$, we have the
  canonical flat Mishchenko line bundle 
  $\mcV(M):= \widetilde M \times_{\pi} C^*\pi \to M$, a Hilbert $C^*\pi$-bundle, { where 
  $C^* \pi$ is the reduced or maximal group $C^*$-algebra for $\pi$,
  respectively}.  {If $M$ is a closed  
  connected 
  spin manifold}, the  Mishchenko-Fomenko index
  of the Dirac operator twisted by $\mcV(M)$, denoted $\ind(D_{\mcV(M)})\in
  K_n(C^*\pi)$, is called the \emph{Rosenberg index} and often written
  $\alpha(M):= \ind(D_{\mcV(M)})$. Here $n=\dim(M)$.  If $M$ is compact, standard
  arguments (see e.g.~\cite[Section 2.1]{PiazzaSchick}) show that this can also be
  viewed as the coarse index $\ind(D_{\mcV(M)})$ applied to the
  \emph{compact} manifold $M$ in which case $C^*(M;C^*\pi)$ is canonically
  Morita equivalent to $C^*\pi$ and therefore has the same K-theory.

  By one of the many possible definitions of the Baum-Connes assembly
  map this is the image of the K-homology class of $B\pi$ represented by $[M]$
  in the Baum-Douglas picture of K-homology under the Baum-Connes assembly
  map.

The following {suspension result} is well known and essentially contained in
\cite{Rosenberg_PSC_NC_I}, compare in particular Proposition 2.9 and its proof
(and the references therein) and the proof of Theorems 2.11 and 3.1 in
\cite{Rosenberg_PSC_NC_I}. 

\begin{proposition}\label{prop:product}
  We have $C^*(\pi\times \integers)=C^*\pi\tensor
  C^*\integers$ and  $K_1(C^*\integers)\cong \integers$.  {The last isomorphism  is induced by the  generator $e = \ind(D_{\mcV(S^1)}) \in K_1(C^* \integers)$, 
  the Rosenberg index of the Dirac operator on $S^1$, where $S^1$ carries the canonical orientation 
  and  any one  of the two
  possible spin structures.}

  For an arbitrary closed spin manifold $M$ we have the product
  formula 
  \begin{multline*}
    \ind(D_{\mcV(M\times S^1)}) = \ind(D_{\mcV(M)}) \tensor e\\
 \in
    K_n(C^*\pi_1(M))\tensor K_1(C^*\integers) \subset
    K_{n+1}(C^*\pi_1(M\times S^1))
  \end{multline*}
  relating the Rosenberg indices of $M\times S^1$ and $M$, where we use the
  inclusion $K_n(A)\tensor K_1(C^*\integers)\hookrightarrow K_{n+1}(A\tensor
  C^*\integers)$ coming from the K\"unneth theorem.
\end{proposition}

We can now state and prove the following result, which implies Theorem \ref{theo:codim2}.

\begin{thm}
\label{thm:main_result}
Let $M$ be a connected closed manifold with $\dim(M)\geq 3$ and
$W\subset M$ a connected submanifold of codimension zero with boundary $\partial W$.
Additionally, assume that the following holds:
\begin{enumerate}[(1)]
\item The boundary $\partial W$ is connected.
\item The second homotopy group of $M$ vanishes: $\pi_2(M)=0$.
\item The Hurewicz map $\mathrm{hur}_1\pp\pi_1(\partial W)\pf H_1(\partial W)$
 is injective when restricted to the kernel
 $\ker{i_*}\subset\pi_1(\partial W)$
 of the map induced by the inclusion map $i\pp\partial W\pf W$.
\item The inclusion map $j\pp W\pf M$ induces a monomorphism
$j_*\colon \pi_1(W)\pf\pi_1(M)$.  
\end{enumerate}

Then the following holds:

\begin{enumerate}[(a)]
\item Let $p\pp\overline{M}\pf M$ be the covering corresponding to 
the subgroup $j_*(\pi_1(W))$ of $\pi_1(M)$, and
$\overline{W}\subset\overline{M}$
be a lift as isometric copy of $W$ to $\overline{M}$, which exists by the
choice of this
covering. Denote by $\mathbf{D}(\overline{M},\overline{W})$
the double of the manifold $\overline{M}\setminus\mathrm{int}(\overline{W})$. This double is partitioned by $\partial\overline{W}$. There exists an extension of the Mish\-chen\-ko line bundle $\mcV({\partial\overline{W}})$
over $\partial\overline{W}$ to a flat bundle $\mcE$ over $\mathbf{D}(\overline{M},\overline{W})$.

\item If $M$ is a spin manifold and $W=N\times\mathbf{D}^2$ is a tubular
  neighborhood of a connected and closed
submanifold $N\subset M$ with $\mathrm{codim}(N)=2$ and trivial normal bundle, then (3) is automatically
satisfied and if $\alpha(N)\neq 0$ then the manifold $M$ does not admit a
metric of positive scalar curvature.
\end{enumerate}
\end{thm}

The condition (2) in Theorem~\ref{thm:main_result} is necessary. For example consider 
$M:=N\times S^2$ with some  $N$ which has non-trivial $\hat{A}$-genus. Then $M$ does
admit a metric with positive scalar curvature, but all the other assumptions
are satisfied by a tubular neighborhood $W$ of one copy of $N$ in $N\times
S^2$.

\begin{remark}
  If $M,N$ are as in part b) of Theorem \ref{thm:main_result}, and $\alpha(N)\ne 0$ then
  the index of $\pi_1(N)$ in $\pi_1(M)$ is necessarily infinite. Otherwise,
  passing to the finite covering $\overline{ M}$ the complement $\overline M\setminus \overline
  W$ is a compact spin bordism between $N\times S^1$ and the empty set, over which the 
  Mishchenko bundle of $N\times S^1$ extends.

  By bordism invariance of the index of the twisted Dirac operator we have 
  $\alpha(N\times S^1)=0$  and therefore also $\alpha(N)=0$, as explained in
  Proposition \ref{prop:product}.
\end{remark}

\begin{cor}
Let $N$ be a closed connected spin manifold with $\pi_2(N)=0$ and $\alpha(N)\neq 0$ in 
$ K_*( C^*\pi_1(N))$. Let $X$ be the total space of a fiber bundle
$N\hookrightarrow X\rightarrow\Sigma$
with fiber $N$ over a compact surface $\Sigma$ different from $S^2$ or
${\mbR}P^2$. 
If the spin structure on $N$ extends to a compatible spin structure on $X$, then $X$ does not 
admit a Riemannian metric with positive scalar curvature.
\end{cor}
\begin{proof}
We view $N$ as fiber over some point in $\Sigma$. Local triviality
of the bundle implies the existence of a {trivialized} tubular neighborhood
$W$ of $N$ in $X$. By assumption $\Sigma$ is neither ${\mbR}P^2$ nor $S^2$. 
Hence $\pi_2(\Sigma)=0$ and the long exact homotopy sequence
of the bundle implies that the inclusion $j\pp W\pf X$ is $\pi_1$-injective.
By the same reasoning $\pi_2(X)=0$. So (2) and (4) of
Theorem~\ref{thm:main_result} are satisfied.
\end{proof}

\begin{proof}[Proof of Theorem~\ref{thm:main_result}]

We consider the connected covering $p\pp\overline{M}\pf M$ corresponding to the
subgroup $j_*(\pi_1(W))$ of $\pi_1(M)$. The inclusion map $j\pp W\pf M$ lifts to an injection $\overline{j}\pp \overline{W}\pf\overline{M}$ which is an $\pi_1$-isomorphism, where $\overline{W}$ is homeomorphic
to $W$ via $p$.

\medskip

(a) We will show subsequently that the inclusion map $k\pp\partial\overline{W}\pf\overline{M}\setminus\overline{W}$ 
induces an injection on $\pi_1$ and that there exists a homomorphism
$r\pp\pi_1(\overline{M}\setminus\overline{W})\pf\pi_1(\partial\overline{W})$
satisfying $r\circ k_*=\id$, i.e., $k$ is a split injection. From this it
follows that $\mcE:=(Br\circ c)^*\mcV(B\pi_1( \partial\overline{W}))$
satisfies $k^*\mcE\cong\mcV(\partial\overline{W})$ if $c$ is the classifying map of the
universal covering of $\overline{M}\setminus\overline{W}$.

\medskip

Injectivity of $k_*$: Let $i\pp\partial\overline{W}\hookrightarrow
\overline{W}$ be the inclusion. Then the diagram
\begin{equation}
\begin{CD}
\label{CD:Useful_diagram_I}
\pi_1(\partial\overline{W}) @>k_*>> \pi_1(\overline{M}\setminus\overline{W}) \\
@V i_* VV  @VV m_*V \\
\pi_1(\overline{W}) @>\overline{j}_*>\cong> \pi_1(\overline{M})
\end{CD}
\end{equation}
commutes (the right vertical arrow is  given by the inclusion $m\pp
\overline{M}\setminus\overline{W}\pf \overline{M}$). Since $\overline{j}$ is a
$\pi_1$-isomorphism one has $\ker{k_*}\subset\ker{i_*}$. Therefore, if
$[\alpha]\in\ker{k_*}$ then the loop $\alpha$ is both null-homotopic as a map
to $\overline{M}\setminus\overline{W}$ and as a map to $\overline{W}$.
This allows us to construct a singular sphere $\sigma\pp S^2\pf\overline{M}$
which maps the lower and upper hemisphere $S^2_-$ and $S^2_+$, into $\overline{W}$
and $\overline{M}\setminus\overline{W}$, respectively, and whose restriction
of $\sigma$ to the equator $S^1\subset S^2$ is $\alpha$. By assumption (2) we
have $\pi_2(\overline M) =\pi_2(M)=0$ and hence $\sigma_*[S^2]=0$ in singular
homology. Therefore, by the construction
of the boundary operator $\partial$ of the Mayer-Vietoris
sequence of the triad
$(\overline{M},\overline{M}\setminus\overline{W},\overline{W})$
also $\partial(\sigma_*[S^2])=\alpha_*[S^1]=\mathrm{hur}_1[\alpha]=0$.
But this in conjunction with (3) implies $[\alpha]=0$, proving that $k_*$ is injective.

\medskip

Existence of $r$: Since $\overline{M}$ and $\overline{W}$ are path connected,
$\pi_1(\overline{W})\to \pi_1(\overline{M})$ is an isomorphism and
$\pi_2(\overline{M})=0$ by assumption (ii), the pair $(\overline M,\overline
W)$ is $2$-connected. By the relative Hurewicz theorem,
$H_*(\overline{M},\overline{W})=0$ for $j=0,1,2$.
 By excision the groups $H_1(\overline{M}\setminus\overline{W},\partial\overline{W})$ and
$H_2(\overline{M}\setminus\overline{W},\partial\overline{W})$ are also trivial. In particular,
$H_1(k)$ is an isomorphism. From~(\ref{CD:Useful_diagram_I}) we obtain the
following diagram
\begin{equation}
\begin{CD}
\label{CD:Useful_diagram_II}
\pi_1(\partial\overline W) @>k_*>>  \pi_1(\overline M\setminus\overline{W})\\
@V i_*\times\mathrm{hur}_1 VV @VV m_*\times\mathrm{hur}_1 V\\
\pi_1(\overline{W}) \times H_1(\partial\overline{W}) @>{\cong}>\overline{j}_*\times H_1(k)>  \pi_1(\overline M)\times H_1(\overline M\setminus\overline{W})
\end{CD}
\end{equation}
which commutes by the naturality of the Hurewicz homomorphism.
The lower horizontal arrow in~(\ref{CD:Useful_diagram_II}) is an isomorphism
as $\overline{j}_*$ and $H_1(k)$ are isomorphisms. Furthermore, our assumption
(3) implies the injectivity of $i_*\times\mathrm{hur}_1$.
This allows us to regard $\pi_1(\partial\overline{W})$ as subgroup of $\pi_1(\overline{M})\times H_1(\overline{M}\setminus\overline{W})$
via the injection given by the composition of this injection with the lower horizontal arrow in
~(\ref{CD:Useful_diagram_II}). The right vertical arrow then surjects onto this subgroup of $\pi_1(\overline{M})\times H_1(\overline{M}\setminus\overline{W})$. Define the map $r\pp\pi_1(\overline M\setminus\overline{W})\pf\pi_1(\partial\overline W)$ by 
\[
r:=(i_*\times\mathrm{hur}_1)^{-1}\circ (\overline{j}_*\times H_1(k))^{-1}\circ (m_*\times\mathrm{hur}_1)\text{ ,}
\]
where we restrict the  target of the injection $i_*\times\mathrm{hur}_1$ 
so that we get an invertible map.
Clearly, $r\circ k_*=\id$.

\medskip

(b) Assume first that $M$ is odd dimensional. Denote by $W$ a trivial tubular neighbourhood of $N$. Then $W$ is a zero-codimensional submanifold of $M$. The manifold $\mathbf{D}(\overline{M},\overline{W})$ admits a spin structure and is partitioned by the boundary $\partial\overline{W}\cong N\times S^1$ of $\overline{W}$. By part (a) there is a flat bundle $\mathcal{E}$ over $\mathbf{D}(\overline{M},\overline{W})$ which extends the Mish\-chen\-ko line bundle $\mcV(\partial\overline{W})$ over $\partial\overline{W}$. By Theorem~\ref{thm:PMIT} we have:
\begin{equation}
\varphi_{\partial\overline{W}}(\ind(D_{\mathbf{D}(\overline{M},\overline{W}),\mcE}))
=
\ind(D_{\partial\overline{W},\mcV(\partial\overline{W})})\in K_0(C^*\pi_1(\partial\overline{W}))\text{ .}
\end{equation}
On the other hand, using Proposition \ref{prop:product},
\begin{equation}
\begin{aligned}
\ind(D_{\partial\overline{W},\mcV(\partial\overline{W})})
&=
\alpha(\partial\overline{W}) =
\alpha(\partial W) \\
&=
\alpha(N\times S^1) =
\alpha(N)\otimes e\text{ .}
\end{aligned}
\end{equation}
Since we assume $\alpha(N)\neq 0$, by Proposition \ref{prop:product}  we can
conclude 
\begin{equation}
\label{equation:conclusion}
\ind(D_{\mathbf{D}(\overline{M},\overline{W}),\mcE})\neq 0.
\end{equation}
{We conclude the proof by contradiction as follows.} If $M$ admits a metric of positive scalar
curvature,  then  $\overline{M}$ admits a metric of uniformly positive
scalar curvature. We can use this metric (deformed in a neighborhood of
$\partial{\overline W}$ to get a smooth metric)
to obtain  a Riemannian metric with
uniformly 
positive scalar curvature outside of a compact neighbourhood of $\partial
\overline{W}$ on $\mathbf{D}(\overline{M},\overline{W})$.
But since the bundle $\mcE$ is flat, Equation \eqref{equation:conclusion} and
Theorem \ref{thm:vanishing_theorem} 
imply that $\mathbf{D}(\overline{M},\overline{W})$ has no metric with
uniformly positive scalar curvature outside
of a compact subset. Hence $M$ cannot admit a metric of positive scalar curvature. 
\medskip

Now assume that $M$ is even-dimensional. In this case we replace the pair
$(M,N)$ by $(M\times S^1,N\times S^1)$. Since $N$ has trivial normal bundle
in $M$ the normal bundle of $N\times S^1$ in $M\times S^1$ is trivial. Also
the fundamental group of the submanifold still injects into the fundamental
group of the ambient manifold. Since
\begin{equation}
\alpha(N\times S^1)\neq 0\Longleftrightarrow \alpha(N)\otimes e\neq 0 \iff
\alpha(N)\neq 0
\end{equation}
it follows from the previous paragraph that $M\times S^1$ admits no metric of positive scalar
curvature. So in particular $M$ has no such metric.
\end{proof}

\begin{remark}
  It should be possible to generalize the results of this paper in the
  following directions:
  \begin{itemize}
  \item Using real $C^*$-operators and $Cl_n$-linear versions, more refined
    invariants in the K-theory of real group $C^*$-algebras should be
    defined for which the same kind of vanishing result holds, and which
    should give rise to stronger obstructions to positive scalar
    curvature. Note that the (stable) Gromov-Lawson-Rosenberg conjecture concerns  the real Dirac operator and the corresponding Rosenberg index.
  \item Using suitable further twists, as developed systematically by Stolz,
    compare e.g.~\cite[Section 5]{RosenbergStolz} one should be able to extend
    the theory to non-spin manifolds and even non-orientable manifolds,
    provided the universal covering remains a spin manifold.
  \item The partitioned manifold index theorem underlying our approach has generalizations to
    multi-partitioned manifolds \cite{SchickZadeh}. It should be possible, at
    least in special, iterated situations, to generalize the codimension two
    obstruction of Theorem \ref{thm:main_result} to even higher
    codimensions. For example, think of the following situation: one is given
    a codimension two hypersurface $N_1$ of a manifold $M$ which itself
    contains a codimension two 
    hypersurface $H$, for example let  $H=N_1\cap N_2$ be the
    intersection of two 
    codimension two hypersurfaces. Is the Rosenberg index of $H$ an
    obstruction to positive scalar curvature of $M$ {(under an appropriate assumption on fundamental groups and the vanishing of higher homotopy groups of $M$)}?
  \end{itemize}

\end{remark}

% \addcontentsline{toc}{section}{References}         %
 \bibliographystyle{plain}
% \bibliography{literature}

\begin{bibdiv}
  \begin{biblist}

\bib{Abramowitz}{book}{
   author={Abramowitz, Milton},
   author={Stegun, Irene A.},
   title={Handbook of mathematical functions with formulas, graphs, and
   mathematical tables},
   series={National Bureau of Standards Applied Mathematics Series},
   volume={55},
   publisher={For sale by the Superintendent of Documents, U.S. Government
   Printing Office, Washington, D.C.},
   date={1964},
   pages={xiv+1046},
   review={\MR{0167642 (29 \#4914)}},
}

\bib{MR2088027}{book}{
   author={Gallot, Sylvestre},
   author={Hulin, Dominique},
   author={Lafontaine, Jacques},
   title={Riemannian geometry},
   series={Universitext},
   edition={3},
   publisher={Springer-Verlag},
   place={Berlin},
   date={2004},
   pages={xvi+322},
   isbn={3-540-20493-8},
   review={\MR{2088027 (2005e:53001)}},
   doi={10.1007/978-3-642-18855-8},
}

\bib{MR2509837}{book}{
   author={Ginoux, Nicolas},
   title={The Dirac spectrum},
   series={Lecture Notes in Mathematics},
   volume={1976},
   publisher={Springer-Verlag},
   place={Berlin},
   date={2009},
   pages={xvi+156},
   isbn={978-3-642-01569-4},
   review={\MR{2509837 (2010a:58039)}},
   doi={10.1007/978-3-642-01570-0},
}
	
\bib{MR720933}{article}{
   author={Gromov, Mikhael},
   author={Lawson, H. Blaine, Jr.},
   title={Positive scalar curvature and the Dirac operator on complete
   Riemannian manifolds},
   journal={Inst. Hautes \'Etudes Sci. Publ. Math.},
   number={58},
   date={1983},
   pages={83--196 (1984)},
   issn={0073-8301},
   review={\MR{720933 (85g:58082)}},
}    
\bib{HankeKotschickRoeSchick}{article}{
   author={Hanke, Bernhard},
   author={Kotschick, Dieter},
   author={Roe, John},
   author={Schick, Thomas},
   title={Coarse topology, enlargeability, and essentialness},
   language={English, with English and French summaries},
   journal={Ann. Sci. \'Ec. Norm. Sup\'er. (4)},
   volume={41},
   date={2008},
   number={3},
   pages={471--493},
   issn={0012-9593},
   review={\MR{2482205 (2009k:58041)}},
}

\bib{MR2259056}{article}{
   author={Hanke, Bernhard},
   author={Schick, Thomas},
   title={Enlargeability and index theory},
   journal={J. Differential Geom.},
   volume={74},
   date={2006},
   number={2},
   pages={293--320},
   issn={0022-040X},
   review={\MR{2259056 (2007g:58024)}},
}
		

\bib{MR2353861}{article}{
   author={Hanke, Bernhard},
   author={Schick, Thomas},
   title={Enlargeability and index theory: infinite covers},
   journal={$K$-Theory},
   volume={38},
   date={2007},
   number={1},
   pages={23--33},
   issn={0920-3036},
   review={\MR{2353861 (2008i:57030)}},
   doi={10.1007/s10977-007-9004-3},
}
		

		

\bib{MR1451755}{article}{
   author={Higson, Nigel},
   author={Pedersen, Erik Kj{\ae}r},
   author={Roe, John},
   title={$C^\ast$-algebras and controlled topology},
   journal={$K$-Theory},
   volume={11},
   date={1997},
   number={3},
   pages={209--239},
   issn={0920-3036},
   review={\MR{1451755 (98g:19009)}},
   doi={10.1023/A:1007705726771},
}
\bib{MR1817560}{book}{
   author={Higson, Nigel},
   author={Roe, John},
   title={Analytic $K$-homology},
   series={Oxford Mathematical Monographs},
   note={Oxford Science Publications},
   publisher={Oxford University Press},
   place={Oxford},
   date={2000},
   pages={xviii+405},
   isbn={0-19-851176-0},
   review={\MR{1817560 (2002c:58036)}},
}
\bib{HR4}{article}{
   author={Higson, Nigel},
   author={Roe, John},
   title={$K$-homology, assembly and rigidity theorems for relative eta
   invariants},
   journal={Pure Appl. Math. Q.},
   volume={6},
   date={2010},
   number={2, Special Issue: In honor of Michael Atiyah and Isadore
   Singer},
   pages={555--601},
   issn={1558-8599},
   review={\MR{2761858 (2011k:58030)}},
   doi={10.4310/PAMQ.2010.v6.n2.a11},
}		

\bib{MR1219916}{article}{
   author={Higson, Nigel},
   author={Roe, John},
   author={Yu, Guoliang},
   title={A coarse Mayer-Vietoris principle},
   journal={Math. Proc. Cambridge Philos. Soc.},
   volume={114},
   date={1993},
   number={1},
   pages={85--97},
   issn={0305-0041},
   review={\MR{1219916 (95c:19006)}},
   doi={10.1017/S0305004100071425},
}
\bib{MR1142484}{article}{
  author={Hilsum, Michel},
  author={Skandalis, Georges},
  title={Invariance par homotopie de la signature \`a coefficients dans un
  fibr\'e presque plat},
  language={French},
  journal={J. Reine Angew. Math.},
  volume={423},
  date={1992},
  pages={73--99},
  issn={0075-4102},
  review={\MR{1142484 (93b:46137)}},
  doi={10.1515/crll.1992.423.73},
}



\bib{MR1949157}{article}{
   author={Kucerovsky, Dan},
   title={Functional calculus and representations of $C_0(\Bbb C)$ on a
   Hilbert module},
   journal={Q. J. Math.},
   volume={53},
   date={2002},
   number={4},
   pages={467--477},
   issn={0033-5606},
   review={\MR{1949157 (2003j:46086)}},
   doi={10.1093/qjmath/53.4.467},
}

\bib{MR1325694}{book}{
   author={Lance, E. Christopher},
   title={Hilbert $C^*$-modules},
   series={London Mathematical Society Lecture Note Series},
   volume={210},
   note={A toolkit for operator algebraists},
   publisher={Cambridge University Press},
   place={Cambridge},
   date={1995},
   pages={x+130},
   isbn={0-521-47910-X},
   review={\MR{1325694 (96k:46100)}},
   doi={10.1017/CBO9780511526206},
}

\bib{MR1031992}{book}{
   author={Lawson, H. Blaine, Jr.},
   author={Michelsohn, Marie-Louise},
   title={Spin geometry},
   series={Princeton Mathematical Series},
   volume={38},
   publisher={Princeton University Press},
   place={Princeton, NJ},
   date={1989},
   pages={xii+427},
   isbn={0-691-08542-0},
   review={\MR{1031992 (91g:53001)}},
}

\bib{MR548506}{article}{
   author={Mi{\v{s}}{\v{c}}enko, Alexander S.},
   author={Fomenko, Anatoly T.},
   title={The index of elliptic operators over $C^{\ast} $-algebras},
   language={Russian},
   journal={Izv. Akad. Nauk SSSR Ser. Mat.},
   volume={43},
   date={1979},
   number={4},
   pages={831--859, 967},
   issn={0373-2436},
   review={\MR{548506 (81i:46075)}},
}
		
\bib{thesis_pape}{thesis}{
      AUTHOR = {Pape, Daniel},
     TITLE = {Index theory and positive scalar curvature},
   SCHOOL  = {Georg-August-Universit{\"a}t G{\"o}ttingen},
     YEAR  = {2011}
}

\bib{PiazzaSchick}{unpublished}{
  author={Piazza, Paolo},
  author={Schick, Thomas},
  title={Rho-classes, index theory and Stolz' positive scalar curvature
sequence}, 
  note={arXiv:1309.4370; to appear in Journal of Topology},
} 


\bib{MR1147350}{article}{
   author={Roe, John},
   title={Coarse cohomology and index theory on complete Riemannian
   manifolds},
   journal={Mem. Amer. Math. Soc.},
   volume={104},
   date={1993},
   number={497},
   pages={x+90},
   issn={0065-9266},
   review={\MR{1147350 (94a:58193)}},
   doi={10.1090/memo/0497},
}

\bib{Roe_AMS}{book}{
   author={Roe, John},
   title={Index theory, coarse geometry, and topology of manifolds},
   series={CBMS Regional Conference Series in Mathematics},
   volume={90},
   publisher={Published for the Conference Board of the Mathematical
   Sciences, Washington, DC},
   date={1996},
   pages={x+100},
   isbn={0-8218-0413-8},
   review={\MR{1399087 (97h:58155)}},
}
		
\bib{MR1670907}{book}{
   author={Roe, John},
   title={Elliptic operators, topology and asymptotic methods},
   series={Pitman Research Notes in Mathematics Series},
   volume={395},
   edition={2},
   publisher={Longman},
   place={Harlow},
   date={1998},
   pages={ii+209},
   isbn={0-582-32502-1},
   review={\MR{1670907 (99m:58182)}},
}
		



\bib{Roe_partial_vanish}{unpublished}{
  author={Roe, John},
  title={Positive curvature, partial vanishing theorems, and coarse indices},
  date={2012},
  note={arXiv:1210.6100},
}

\bib{Rosenberg_PSC_NC_I}{article}{
   author={Rosenberg, Jonathan},
   title={$C^{\ast} $-algebras, positive scalar curvature, and the Novikov
   conjecture},
   journal={Inst. Hautes \'Etudes Sci. Publ. Math.},
   number={58},
   date={1983},
   pages={197--212 (1984)},
   issn={0073-8301},
   review={\MR{720934 (85g:58083)}},
}
		
\bib{Rosenberg_PSC_NC_II}{article}{
   author={Rosenberg, Jonathan},
   title={$C^\ast$-algebras, positive scalar curvature and the Novikov
   conjecture. II},
   conference={
      title={Geometric methods in operator algebras},
      address={Kyoto},
      date={1983},
   },
   book={
      series={Pitman Res. Notes Math. Ser.},
      volume={123},
      publisher={Longman Sci. Tech.},
      place={Harlow},
   },
   date={1986},
   pages={341--374},
   review={\MR{866507 (88f:58140)}},
}
		

\bib{Rosenberg_PSC_NC_III}{article}{
   author={Rosenberg, Jonathan},
   title={$C^\ast$-algebras, positive scalar curvature, and the Novikov
   conjecture. III},
   journal={Topology},
   volume={25},
   date={1986},
   number={3},
   pages={319--336},
   issn={0040-9383},
   review={\MR{842428 (88f:58141)}},
   doi={10.1016/0040-9383(86)90047-9},
}
\bib{RosenbergStolz}{article}{
   author={Rosenberg, Jonathan},
   author={Stolz, Stephan},
   title={Metrics of positive scalar curvature and connections with surgery},
   conference={
      title={Surveys on surgery theory, Vol. 2},
   },
   book={
      series={Ann. of Math. Stud.},
      volume={149},
      publisher={Princeton Univ. Press},
      place={Princeton, NJ},
   },
   date={2001},
   pages={353--386},
   review={\MR{1818778 (2002f:53054)}},
}
		
\bib{MR1235284}{article}{
   author={Rosenberg, Jonathan},
   author={Weinberger, Shmuel},
   title={Higher $G$-signatures for Lipschitz manifolds},
   journal={$K$-Theory},
   volume={7},
   date={1993},
   number={2},
   pages={101--132},
   issn={0920-3036},
   review={\MR{1235284 (94j:58165)}},
   doi={10.1007/BF00962083},
}
			
\bib{MR1632971}{article}{
   author={Schick, Thomas},
   title={A counterexample to the (unstable) Gromov-Lawson-Rosenberg
   conjecture},
   journal={Topology},
   volume={37},
   date={1998},
   number={6},
   pages={1165--1168},
   issn={0040-9383},
   review={\MR{1632971 (99j:53049)}},
   doi={10.1016/S0040-9383(97)00082-7},
}

\bib{SchickZadeh}{unpublished}{
  author={Schick, Thomas},
  author={Zadeh, Mostafa Esfahani},
  title={ Large scale index of multi-partitioned manifolds},
  note={arXiv:1308.0742},
  url={http://arxiv.org/abs/1308.0742},
}

\bib{MR535700}{article}{
   author={Schoen, Richard},
   author={Yau, Shing-Tung},
   title={On the structure of manifolds with positive scalar curvature},
   journal={Manuscripta Math.},
   volume={28},
   date={1979},
   number={1-3},
   pages={159--183},
   issn={0025-2611},
   review={\MR{535700 (80k:53064)}},
   doi={10.1007/BF01647970},
}
	 
\bib{Schroedinger}{article}{
    Author = {{Schr\"odinger}, Erwin},
    Title = {{Diracsches Elektron im Schwerefeld. I.}},
    Journal = {{Sitzungsber. Preu{\ss}. Akad. Wiss., Phys.-Math. Kl.}},
    Volume = {1932},
    Pages = {105--128},
    Year = {1932},
    Publisher = {Preu{\ss}ische Akademie der Wissenschaften, Berlin},
}


\bib{MR1937026}{article}{
   author={Stolz, Stephan},
   title={Manifolds of positive scalar curvature},
   conference={
      title={Topology of high-dimensional manifolds, No. 1, 2},
      address={Trieste},
      date={2001},
   },
   book={
      series={ICTP Lect. Notes},
      volume={9},
      publisher={Abdus Salam Int. Cent. Theoret. Phys., Trieste},
   },
   date={2002},
   pages={661--709},
   review={\MR{1937026 (2003m:53059)}},
}
		

\bib{MR1403963}{article}{
   author={Stolz, Stephan},
   title={Positive scalar curvature metrics---existence and classification
   questions},
   conference={
      title={ 2},
      address={Z\"urich},
      date={1994},
   },
   book={
      publisher={Birkh\"auser},
      place={Basel},
   },
   date={1995},
   pages={625--636},
   review={\MR{1403963 (98h:53063)}},
}

\bib{Vassout}{thesis}{
  AUTHOR = {Vassout, St\'{e}phane},
  TITLE = {Feuilletages et r\'{e}sidu non commutatif longitudinal},
  SCHOOL = {Universit\'{e} Pierre et Marie Curie -- Paris VI},
  YEAR = {2001},
  note = {Available at www.imj-prg.fr/$\sim$stephane.vassout/}
}

\bib{MR1435703}{article}{
   author={Yu, Guoliang},
   title={$K$-theoretic indices of Dirac type operators on complete
   manifolds and the Roe algebra},
   journal={$K$-Theory},
   volume={11},
   date={1997},
   number={1},
   pages={1--15},
   issn={0920-3036},
   review={\MR{1435703 (98e:19002)}},
   doi={10.1023/A:1007706112341},
}
		
\bib{MR2670972}{article}{
   author={Zadeh, Mostafa Esfahani},
   title={Index theory and partitioning by enlargeable hypersurfaces},
   journal={J. Noncommut. Geom.},
   volume={4},
   date={2010},
   number={3},
   pages={459--473},
   issn={1661-6952},
   review={\MR{2670972 (2011i:58034)}},
   doi={10.4171/JNCG/63},
}
\bib{MR2929034}{article}{
   author={Zadeh, Mostafa Esfahani},
   title={A note on some classical results of Gromov-Lawson},
   journal={Proc. Amer. Math. Soc.},
   volume={140},
   date={2012},
   number={10},
   pages={3663--3672},
   issn={0002-9939},
   review={\MR{2929034}},
   doi={10.1090/S0002-9939-2012-11544-8},
}
  \end{biblist}
\end{bibdiv}

\end{document}